%% file: mainthesis.tex
\documentclass[12pt]{thesis}  
\usepackage{titlesec}
   \titleformat{\chapter}
      {\normalfont\large}{Chapter \thechapter:}{1em}{}

\usepackage{graphicx}
\usepackage{cite}
\usepackage{lscape}
\usepackage{indentfirst}
\usepackage{latexsym}
\usepackage{multirow}
\usepackage{tabls}
\usepackage{wrapfig}
\usepackage{slashbox}
\usepackage{longtable}
\usepackage{supertabular}
\usepackage{subfigure}
\usepackage{dsfont}
\usepackage{amsmath,amsthm,amssymb,amsfonts, amscd}
\usepackage{tikz}
\usetikzlibrary{matrix}
\newcommand{\biddots}{
\begin{tikzpicture}
\filldraw [black] (0,0) circle (1.8pt);
\filldraw [black] (.2,.2) circle (1.8pt);
\filldraw [black] (.4,.4) circle (1.8pt);
\end{tikzpicture}
}

\newcommand{\C}{\mathbb{C}}
\newcommand{\R}{\mathbb{R}}
\newcommand{\SO}{\mathrm{SO}}
\newcommand{\Sp}{\mathrm{Sp}}
\newcommand{\OO}{\mathrm{O}}
\newcommand{\GL}{\mathrm{GL}}

\newcommand{\vol}{\mathrm{vol}}
\newcommand{\wwedge}[1]{\sideset{}{^{#1}}\bigwedge}
\newcommand{\bfv}{\mathbf{v}}
\newtheorem{thm}[subsection]{Theorem}  
\newtheorem{lem}[subsection]{Lemma}         
\newtheorem*{lem*}{Lemma}         
\newtheorem{prop}[subsection]{Proposition}
\newtheorem*{prop*}{Proposition}
\newtheorem{rmk}[subsection]{Remark}
\theoremstyle{definition}
\newtheorem{defn}[subsection]{Definition}   

\newcommand{\Hom}{\mathrm{Hom}}
\newtheorem{cor}[subsection]{Corollary}
\newcommand{\Z}{\mathbb{Z}}
\renewcommand{\baselinestretch}{2}
\begin{document}

\include{Abstract} 
\include{Titlepage} 

\pagestyle{plain}
\pagenumbering{roman}
\setcounter{page}{2}

\include{Acknowledgements} 

\renewcommand{\baselinestretch}{1}
\small\normalsize
\tableofcontents 

\newpage
\setlength{\parskip}{0em}
\renewcommand{\baselinestretch}{2}
\small\normalsize

\setcounter{page}{1}
\pagenumbering{arabic}
\include{Chapter1}

\include{Chapter2}

\include{Chapter3}

\include{Chapter4}

\renewcommand{\baselinestretch}{1}
\small\normalsize
\include{Bibliography}

\end{document}

%% file: Abstract.tex

\hbox{\ }

\renewcommand{\baselinestretch}{1}
\small \normalsize

\begin{center}
\large{{ABSTRACT}} 

\vspace{3em} 

\end{center}
\hspace{-.15in}
\begin{tabular}{ll}
Title of dissertation:    & {\large  The Relative Lie Algebra Cohomology} \\
& {\large of the Weil Representation }\\

\ \\
&                          {\large Jacob Ralston, Doctor of Philosophy, 2015} \\
\ \\
Dissertation directed by: & {\large  Professor John Millson} \\
&  				{\large	 Department of Mathematics } \\
\end{tabular}

\vspace{3em}

\renewcommand{\baselinestretch}{2}
\large \normalsize

We study the relative Lie algebra cohomology of $\mathfrak{so}(p,q)$ with values in the Weil representation $\varpi$ of the dual pair $\mathrm{Sp}(2k, \R) \times \OO(p,q)$.  Using the Fock model defined in Chapter \ref{Introchapter}, we filter this complex and construct the associated spectral sequence. We then prove that the resulting spectral sequence converges to the relative Lie algebra cohomology and has  $E_0$ term, the associated graded complex, isomorphic to a Koszul complex, see Section \ref{defofkoszulsection}.  It is immediate that the construction of the spectral sequence of Chapter \ref{spectralchapter} can be applied to any reductive subalgebra $\mathfrak{g} \subset \mathfrak{sp}(2k(p+q), \R)$.  By the Weil representation of $\OO(p,q)$, we mean the twist of the Weil representation of the two-fold cover $\widetilde{\OO(p,q)}$ by a suitable character. We do this to make the center of $\widetilde{\OO(p,q)}$ act trivially.  Otherwise, all relative Lie algebra cohomology groups would vanish, see Proposition \ref{genuineprop}.

In case the symplectic group is large relative to the orthogonal group ($k \geq pq$), the $E_0$ term is isomorphic to a Koszul complex defined by a regular sequence, see \ref{defofkoszulsection}.  Thus, the cohomology vanishes except in top degree.  This result is obtained without calculating the space of cochains and hence without using any representation theory.  On the other hand, in case $k < p$, we know the Koszul complex is not that of a regular sequence from the existence of the class $\varphi_{kq}$ of Kudla and Millson, see \cite{KM2}, a nonzero element of the relative Lie algebra cohomology of degree $kq$.

For the case of $\SO_0(p,1)$ we compute the cohomology groups in these remaining cases, namely $k < p$.  We do this by first computing a basis for the relative Lie algebra cochains and then splitting the complex into a sum of two complexes, each of whose $E_0$ term is then isomorphic to a Koszul complex defined by a regular sequence.

This thesis is adapted from the paper, \cite{BMR}, this author wrote with his advisor John Millson and Nicolas Bergeron of the University of Paris.

%% file: Titlepage.tex

\thispagestyle{empty}
\hbox{\ }
\vspace{1in}
\renewcommand{\baselinestretch}{1}
\small\normalsize

\begin{center}

\large{{The Relative Lie Algebra Cohomology of the Weil Representation \\}}

\ \\
\ \\
\large{by} \\
\ \\
\large{Jacob Ralston}
\ \\
\ \\
\ \\
\ \\
\normalsize
Dissertation submitted to the Faculty of the Graduate School of the \\
University of Maryland, College Park in partial fulfillment \\
of the requirements for the degree of \\
Doctor of Philosophy \\
2015
\end{center}

\vspace{7.5em}

\noindent Advisory Committee: \\
Professor John Millson, Chair/Advisor \\
Professor Jeffrey Adams\\
Professor Thomas Haines\\
Professor Steve Halperin\\
Professor P.S. Krishnaprasad

%% file: Acknowledgements.tex

\renewcommand{\baselinestretch}{2}
\small\normalsize
\hbox{\ }
 
\vspace{-.65in}

\begin{center}
\large{Acknowledgments} 
\end{center} 

\vspace{1ex}

I would like to thank the numerous people who made this thesis possible.  First, I want to thank my advisor and friend John Millson.  Finding an advisor can be a tricky process, but eventually we found ourselves simultaneously in need of a student and an advisor (respectively).  I can honestly say that without an advisor and collaborator like John I could not have completed (or maybe even started) my thesis.  I would also like to thank Matei Machedon for pointing out the importance of the complex $C_-$, Steve Halperin for the proof of Proposition \ref{Steveprop1}, Jeff Adams for answering the many questions John and I had, and Yousheng Shi for his help with constructing the spectral sequence associated to a filtration.

There are many other people without whom I could not have finished (or started!).  Chronologically (and perhaps in other ways too), Paul Koprowski gets a lot of credit.  I met him Freshman year of college and he unknowlingly led the way: squash, math, phsycis, UMD, the UClub, and moving to DC (but hopefully not to Baltimore!).  He assured Chris Laskowski (the graduate chair at the time of my admittance) that I was up to snuff - he may or may not have been correct.  Thanks Paul!

Then there are those I met when I first arrived.  Mike, Scott, Dan, Tim, Lucia, Jon (aka Rubber Duck), Zsolt, the list goes on...  You were all essential, so thank you.  Adam, Rebecca, Ryan, and Sean, thanks for always having your office door open.

My time at Haverford helped shape me in many ways, including mathematically.  I want to thank Steve Wang for being my informal math advisor and Dave Lippel for being my math advisor.  I also want to thank Jon Lima for being there with me (and for staying up as late as necessary to finish problem sets).  I'm not sure who made the right decision, but at least we're both done.

Finally, I would like to thank my family.  You were always there to help me make decisions and let me sleep.  Maybe what you're looking at now answers your favorite question, ``what do you study?"  Probably it doesn't, but that's ok.

%% file: Chapter1.tex

\renewcommand{\thechapter}{1}

\chapter{Introduction}

\subsection{Results for $\SO(p,q) \times \mathrm{Sp}(2k, \R)$ for large $k$ or large $p$.}

We let $(V, (\ , \ ))$ be $\R^{p,q}$, a real vector space of signature $(p,q)$.  We will consider the connected real Lie group $G = \SO_0(p,q)$ with Lie algebra $\mathfrak{so}(p,q)$ and maximal compact subgroup $K = \SO(p) \times \SO(q)$.  Let $\mathcal{P}_k$ be the space of all holomorphic polynomials on $(V \otimes \C)^k$, see Section \ref{weilrepsection}.  We will consider the Weil representation with values in $\mathcal{P}_k$.

We first summarize our results obtained for general $p,q$.  In what follows, the $q_{\alpha \mu}$ are the quadratic polynomials on $(V \otimes \C)^k$ defined in Equation \eqref{qalphadef} and $\vol$ is defined in Equation \eqref{voldefn}.

\subsubsection{Results for large $k$.}

\begin{thm} \label{introtheoremlargek}
Assume $k \geq pq$.  Then we have 
\begin{equation*}
H^\ell(\mathfrak{so}(p,q), \SO(p) \times \SO(q); \mathcal{P}_k) = \begin{cases}
0 &\text{ if } \ell \neq pq \\
(\mathcal{P}_k/(q_{1, p+1}, \ldots, q_{p,p+q}))^K \vol &\text{ if } \ell = pq.
\end{cases}
\end{equation*}
\end{thm}

In fact, we show that the top degree cohomology never vanishes.  That is, $H^{pq}(\mathfrak{so}(p,q), \SO(p) \times \SO(q); \mathcal{P}_k) \neq 0$.  The key point of this theorem is the vanishing in all other degrees.

\begin{rmk}
There are some known vanishing results, for example that the relative Lie algebra cohomology will vanish in all degrees below the rank of the group, in this case $q$, see for example \cite{BW}.  The above theorem implies vanishing below degree $q$.
\end{rmk}

\subsubsection{Results for large $p$.}

\begin{thm}
If $p \geq kq$, then
\begin{equation*}
H^\ell(\mathfrak{so}(p,q), \SO(p) \times \SO(q); \mathcal{P}_k) = \begin{cases}
\text{nonzero } &\text{ if } \ell = kq, pq \\
\text{unknown } &\text{ if } kq < \ell < pq \\
0 &\text{ otherwise}.
\end{cases}
\end{equation*}

\end{thm}

\begin{rmk}
The nonvanishing of the cohomology class of $\varphi_{kq}$, defined in Equation \eqref{defnofKMkq}, in cohomological degree $kq$, was already known but this nonvanishing required the general theorem of Kudla and Millson that $\varphi_{kq}$ is Poincar\'{e} dual to a special cycle.  Then one also has to show the special cycle is nonzero in homology.  The proof in this thesis is the first local proof of nonvanishing.
\end{rmk}

\subsection{Results for $\SO_0(n,1)$ for all $\ell$ and $k$.}

For the case of $\SO_0(n,1)$, we compute the cohomology groups\\
$H^\ell(\mathfrak{so}(n,1), \SO(n); \mathcal{P}_k)$ in the remaining cases, namely $k < n$.  In the following theorem, let $\varphi_k$ be the cocycle constructed in the work of Kudla and Millson, \cite{KM2}, see Section \ref{notation}, Equation \eqref{vaprhikdef}.  In what follows, $c_1, \ldots, c_k$ are the cubic polynomials on $(V \otimes \C)^k$ defined in Equation \eqref{cubic}, $q_1, \ldots, q_n$ are the quadratic polynomials on $(V \otimes \C)^k$ defined in Equation \eqref{qalphadef}, and $\vol$ is as defined in Equation \eqref{voldefn}.  Note that statement $(3)$ of the following theorem is a consequence of Theorem \ref{introtheoremlargek}.

\begin{thm}\label{main}
\hfill

\begin{enumerate}
\item  If $k < n$ then
\begin{equation*}
H^\ell \big(\mathfrak{so}(n,1) ,\SO(n); \mathcal{P}_k \big) = \begin{cases}
\mathcal{R}_k \varphi_k &\text{ if } \ell=k \\
\mathcal{S}_k/(c_1, \ldots, c_k) \vol &\text{ if } \ell=n \\
0 &\text{ otherwise }
\end{cases}
\end{equation*}
\item If $k = n$ then
\begin{equation*}
H^\ell \big(\mathfrak{so}(n,1) ,\SO(n); \mathcal{P}_k \big) = \begin{cases}
\mathcal{R}_k \varphi_k \oplus \mathcal{S}_k/(c_1, \ldots, c_k) \vol &\text{ if } \ell=n \\
0 &\text{ otherwise }
\end{cases}
\end{equation*}
\item If $k>n$
\begin{equation*}
H^\ell \big(\mathfrak{so}(n,1) ,\SO(n); \mathcal{P}_k \big) = \begin{cases}
\big( \mathcal{P}_k/(q_1, \ldots, q_n) \big)^K \vol &\text{ if } \ell=n \\
0 &\text{ otherwise }
\end{cases}
\end{equation*}
\end{enumerate}
\end{thm}

The chart at the top of the next page summarizes Theorem \ref{main}.  The symbol $\bullet$ means the corresponding group is non-zero.

The cohomology groups $H^\ell \big(\mathfrak{so}(n,1) ,\SO(n); \mathcal{P}_k \big)$ are $\mathfrak{sp}(2k, \R)$-modules.  We do not prove this here, but we will now describe these modules.  The key point is that the {\bf cohomology class} of $\varphi_{k}$ is a lowest weight vector for $\mathfrak{sp}(2k, \R)$.  If $k < \frac{n+1}{2}$, then as an $\mathfrak{sp}(2k, \R)$-module $\mathcal{R}_k \varphi_k$ is isomorphic to the space of $\mathrm{MU}(k)$-finite vectors in the holomorphic discrete series representation with parameter $(\frac{n+1}{2}, \cdots,\frac{n+1}{2}) $. If $k < n$, then the cohomology group $H^k\big(\mathfrak{so}(n,1) ,\SO(n); \mathcal{P}_k \big)$ is an irreducible holomorphic representation because it was proved in \cite{KM2} that the class of $\varphi_k$ is a lowest weight vector in $H^k\big(\mathfrak{so}(n,1) ,\SO(n); \mathcal{P}_k \big)$.  On the other hand, the cohomology group \\
$H^n\big(\mathfrak{so}(n,1) ,\SO(n); \mathcal{P}_k \big)$ is never irreducible.  Indeed, if $k < n$ then \\
$H^n\big(\mathfrak{so}(n,1) ,\SO(n); \mathcal{P}_k \big)$ is the direct sum of two nonzero $\mathfrak{sp}(2k,\R)$-modules $H_+^n$ and $H_-^n$ and if $k = n$, then $H^n\big(\mathfrak{so}(n,1) ,\SO(n); \mathcal{P}_k \big)$ is the direct sum of three nonzero $\mathfrak{sp}(2k,\R)$-modules $H_+^n$, $H_-^n$, and $\mathcal{R}_k(V) \varphi_k$.

\begin{figure}
\begin{tikzpicture}
  \matrix (m) [matrix of math nodes,
    nodes in empty cells,nodes={minimum width=3ex,
    minimum height=2ex,outer sep=-5pt},
    column sep=1ex,row sep=1ex]{
    &&&&&&&&& \\
          n   &  &  \bullet & \bullet &  \bullet \bullet \bullet & \bullet & \bullet & \bullet & \bullet & \bullet \bullet \bullet & \\
          n-1 &    &  0 & 0 &  \cdots & 0 & \bullet &0 & 0 & 0 & \\
          n-2 &    &  0 & 0 &  \cdots  & \bullet &0 &0 & 0 & 0 & \\
          \vdots  &   &  0 & 0 &  \biddots &0 & 0 &0 & \vdots & \vdots & \\
          2   &  &  0 & \bullet &  \cdots &0 & 0 &0 & 0 & 0 & \\
          1  &   &  \bullet & 0 &  \cdots &0 & 0 &0 & 0 & 0 & \\
          \ell =0   &  &  0  & 0 &  \cdots &0 & 0 &0 & 0 & 0 & \\
    \quad\strut &  & k=1  &  2  &  \cdots &n-2 & n-1 &  n  & n+1 & \cdots & \strut \\};
\draw[thick] (m-1-2.east) -- (m-9-2.east) ;
\draw[thick] (m-9-1.north) -- (m-9-11.north) ;
\end{tikzpicture}
$$H^\ell(\mathfrak{so}(n,1), \SO(n); \mathcal{P}_k)$$
\end{figure}

\subsection{Outline of paper}
The main results of this paper are all proved in a similar fashion.  We will filter the relative Lie algebra complex (to be referred to as ``the complex'' from now on) using a filtration induced by polynomial degree in Section \ref{constructionofss}.  There is then a spectral sequence associated to any, and hence this, filtered complex, see Section \ref{spectralsection}.  We then define what a Koszul complex is in Section \ref{defofkoszulsection}.  We see that the $E_0$ page of our spectral sequence is isomorphic to a Koszul complex in Subsection \ref{E_0isakoszulsection}.  There are results about the cohomology of a Koszul complex when the defining elements of the Koszul complex form a regular sequence, see Section \ref{defofkoszulsection}.  Then we use general spectral sequence facts, see Section \ref{consequencesofE_1vanishsection}, to deduce facts about the relative Lie algebra cohomology from the cohomology of the $E_0$ term, a Koszul complex.

\subsection{The theta correspondence}

We now explain why it is important to compute the relative Lie algebra cohomology groups with values in the Weil representation for the study of the cohomology of arithmetic quotients of the associated locally symmetric spaces. The key point is that {\it cocycles} of degree $\ell$ with values in the Weil representation of $\mathrm{Sp}(2k,\R) \times \mathrm{O}(p,q)$ give rise (using the theta distribution $\theta$) to {\it closed} differential $\ell$-forms on arithmetic locally symmetric spaces associated to such groups $G$.  This construction gives rise to a map $\theta$ from the relative Lie algebra cohomology of $G$ with values in the Weil representation to the ordinary cohomology of suitable arithmetic quotients $M$ of the symmetric space associated to $G$.  This reduces the {\it global} computation of the cohomology of $M$ to local algebraic computations in $\wwedge{\bullet} \mathfrak{p}^* \otimes \mathcal{P}_k$.  For all cohomology classes studied so far, the span of their images under $\theta$ is the span of the Poincar\'{e} duals of certain totally geodesic cycles in the arithmetic quotient, the ``special cycles'' of Kudla and Millson.  Furthermore, the refined Hodge projection of the map $\theta$ has been shown in \cite{BMM1} and \cite{BMM2} to be onto a certain refined Hodge type for low degree cohomology.  In particular, for $\mathrm{Sp}(2k,\R) \times \OO(n,1)$, it is shown in \cite{BMM1} that this map is onto $H^\ell(M)$ for $\ell < \frac{n}{3}$.  We now describe such a map in more detail. 

We introduce the Schr\"odinger model to describe one such map more easily.  Let $\mathcal{S}(V^k)$ be the space of rapidly decaying functions on the real vector space $V^k$.  Then there is a map, the Bargmann transform $\beta$, from $\mathcal{P}_k \to \mathcal{S}(V^k)$, see \cite{Folland} pages 39-40.  It sends $1$ to the Gaussian and, more generally, $\mathcal{P}_k$ to the span of the hermite functions.

Given a discrete subgroup $\Gamma \subset G$, let $M$ be the associated locally symmetric space $M = \Gamma \backslash G / K$ and $\pi: D \rightarrow M$ be the quotient map.  Take $f \in \mathcal{P}_k$, $\Gamma \subset G$ and a $\Gamma$-stable lattice $\mathcal{L} \subset V^k$.  Then define $\theta_\mathcal{L}$ to be the map which sends $f$ to $\beta(f)$ and then sums its values on the lattice.  That is,
\begin{equation*}
\theta_\mathcal{L}(f) = \sum_{\ell \in \mathcal{L}} \beta(f)(\ell).
\end{equation*}
Then this map is a $\Gamma$-invariant linear functional on $\mathcal{P}_k$,
\begin{equation*}
\theta_\mathcal{L}: \mathcal{P}_k \to \C.
\end{equation*}
Thus, we have the sequence of maps
\begin{equation*}
C^k(\mathfrak{g}, K; \mathcal{P}_k) \cong A^{\bullet}(D, \mathcal{P}_k)^G \hookrightarrow A^{\bullet}(D, \mathcal{P}_k)^\Gamma \xrightarrow{(\theta_\mathcal{L})_*} A^{\bullet}(D, \R)^\Gamma = A^{\bullet}(\Gamma \backslash D),
\end{equation*}
which induces a map on cohomology
\begin{equation} \label{thetastar}
(\theta_\mathcal{L})_* : H^\ell(\mathfrak{g},K;\mathcal{P}_k) \to H^\ell(M, \C).
\end{equation}

The main goal is then to find more classes in $H^\ell(\mathfrak{g},K;\mathcal{P}_k)$ which map, under appropriate choices $\theta_*$, to nonzero classes in $H^\ell(M, \C)$.  It follows from the theory of the theta correspondence that for $\omega \in H^\ell(\mathfrak{g}, K; \mathcal{P}_k)$, we have ${\theta_\mathcal{L}}_*(\omega)$ is an automorphic form on $\Gamma^\prime \backslash \mathrm{Sp}(2k, \R) / \mathrm{U}(n)$ with values in $H^\ell(M, \C)$ for $\Gamma^\prime \subset \mathrm{Sp}(2k, \R)$ a suitable lattice, see \cite{KM2}.

We will describe some features of the image of the map (actually maps) $\theta_*$ in subsubsection \ref{Johnmotivation}.

\subsection{Motivation}

The relative Lie algebra cocycles of Kudla and Millson, $\varphi_{kq}$ and $\varphi_{kp}$ (and their analogues for $\mathrm{U}(p,q)$), and their images under the theta map (actually, theta maps) have been studied for over thirty years.  The motivation is from geometry and we discuss here one of their geometric properties.  Let $\varphi_{q} \in H^{q}(\mathfrak{so}(p,q), \SO(p) \times \SO(q); \mathcal{P}_k)$ where $V = \R^{p,q}$ be given by (see Section \ref{weilrepsection} for definitions of $z, \omega$)
\begin{equation*}
\varphi_{q} = \sum_{1 \leq \alpha_1, \ldots, \alpha_q \leq p} z_{\alpha_1} \cdots z_{\alpha_q} \omega_{\alpha_1, p+1} \wedge \cdots \wedge \omega_{\alpha_q, p+q}.
\end{equation*}

Then define $\varphi_{kq} \in H^{kq}(\mathfrak{so}(p,q), \SO(p) \times \SO(q); \mathcal{P}_k)$ by taking the ``outer'' wedge (see Equation \eqref{outerwedgedef}) of $\varphi_q$ with itself $k$-times
\begin{equation} \label{defnofKMkq}
\varphi_{kq} = \varphi_{q} \wedge \cdots \wedge \varphi_{q}.
\end{equation}

We say a vector $\mathbf{x} = (x_1, \ldots, x_k) \in V^k$ has positive length if the Gram matrix (the matrix of inner products $\beta = (x_i, x_j)$) is positive definite.  Now let $\mathbf{x} \in V(\mathbb{Q})^k$ be a rational vector of positive length.  Let $\Gamma$ be a discrete subgroup of $\SO_0(p,q)$ and $\Gamma_\mathbf{x}$ be the stabilizer of $\mathbf{x}$.

Recall the symmetric space $D$ is the subset of the Grassmannian $\mathrm{Gr}_q(V)$ given by
$$D = \{Z \subset \mathrm{Gr}_q(V) : (\ ,\ )|_Z << 0\}.$$
Let $Z_0 \in D$ be the subspace spanned by the $e_{p+1}, \ldots, e_{p+q}$.  We define the totally geodesic subsymmetric space $D_\mathbf{x} \subset D$ by
$$D_\mathbf{x} = \{ Z \in D : Z \perp \mathrm{span}(\mathbf{x}) \}.$$
Finally, we define the special cycle $C_\mathbf{x}$ by
$$C_\mathbf{x} = \pi(D_\mathbf{x}).$$

We let $\tilde{\varphi} \in A^{\ell}(D, \mathcal{V})^G$ be the image of $\varphi \in C^{\ell}(\mathfrak{g}, K; \mathcal{V})$ under the isomorphism of Proposition \ref{derhamisomorphism}.
Then define 
\begin{equation} \label{specialforms}
\phi_{kq}(\mathbf{x}, Z) = \sum_{\gamma \in \Gamma_\mathbf{x} \backslash \Gamma} \tilde{\varphi}_{kq} (\gamma^{-1} \mathbf{x}, Z).
\end{equation}

Then we have the following theorem of Kudla and Millson, \cite{KM2}.

\begin{thm}
Fix $\mathbf{x} \in V^k$ of positive length.  Then 
\begin{enumerate}
\item $\phi_{kq}(\mathbf{x}, Z)$ extends to a non-holomorphic Siegel modular form $\phi_{kq}(\mathbf{x}, Z, \tau)$ of weight $\frac{p+q}{2}$ for $\tau$ in the Siegel space $\mathbb{S}_g$.
\item $\phi_{kq}(\mathbf{x}, Z, \tau)$ is a closed differential form in the $Z$ variable. 
\item The cohomology class of $\phi_{kq}(\mathbf{x}, Z, \tau)$ is a nonzero multiple, depending on $\mathbf{x}$ and $\tau$, of the Poincar\`{e} dual of the special cycle $C_\mathbf{x}$.
\end{enumerate}

\end{thm}

\begin{rmk}
The closed form $\phi_{kq}(\mathbf{x},Z,\tau)$ does not depend holomorphically on $\tau$, however its cohomology class is holomorphic in $\tau$.
\end{rmk}



\subsubsection{The classes $\phi_{kq}(\mathbf{x}, Z, \tau)$ and the subspace of the cohomology they span} \label{Johnmotivation}

In what follows we  will ignore the fact that to get cocompact subgroups of $\SO(p,q)$ for $p+q>4$ we must choose a totally-real number field $E$ and use restriction of scalars from $E$ to $\mathbb{Q}$.  Hence for $p+q>4$ the manifolds $M$ defined below will have finite volume but will not be compact.  We leave the required modifications to the expert reader.

We define a family of discrete subgroups $\Gamma$ depending on a choice of positive integer $N$ and a lattice $L$ in $V^k$.  Let $L \in V^k$ be a lattice and $N$ be a positive integer. Then $NL$ is the sublattice of all lattice vectors divisible by $N$.   We then have the finite quotient group $L/NL$.  Let $G(L)$ denote the subgroup of $G$ that stabilizes the lattice $L$.  Then $G$ also stabilizes the sublattice $NL$ and consequently we have a homomorphism
$$\pi: G(L) \to \mathrm{Aut}(L/NL).$$
\begin{defn}
We define $\Gamma$ to be the kernel of $\pi$, that is $\Gamma$ is the congruence subgroup of $G(L)$ of level $N$.
\end{defn}

Recall that $M$ is the quotient $M= \Gamma \backslash D$.  Recall the family of closed forms $\phi_{kq}(\mathbf{x},Z,\tau)$ on $M$ associated to the cocycle $\varphi_{qk}$ defined in Equation \eqref{specialforms}.  Fix $\mathbf{x} \in L^k$ with $(\mathbf{x}, \mathbf{x})$ positive definite.


We now  show that for all $\mathbf{x}, \tau$  the classes $[\phi(\mathbf{x},Z,\tau)]$ lie in the subspace (refined Hodge summand) $SH^{kq}(M)$ of $H^{kq}(M)$ to be defined immediately below.

What follows is an expanded version of the discussion in \cite{BMM1}, page 7.  In the following discusssion, recall we have the splitting (orthogonal for the Killing form) $\mathfrak{g} = \mathfrak{k} \oplus \mathfrak{p}_0$, and let $\mathfrak{p}$ denote the complexification of $\mathfrak{p}_0$.


Since $\mathfrak{p} \cong V_+ \otimes V_- \otimes \C$, where $V$ splits naturally as $V_+ \oplus V_-$,  the group $\mathrm{SL}(q,\C) \cong \mathrm{SL}(V_- \otimes \C)$ acts on $\mathfrak{p}$ and hence it  acts on $\wwedge{j}(\mathfrak{p})$ for all $j$.

\begin{defn} 
The invariant summand $S \wwedge{j} (\mathfrak{p})$ is defined by 
$$ S \wwedge{j} (\mathfrak{p})  = \big( \wwedge{j} (\mathfrak{p}) \big)^{\mathrm{SL}(q, \C)}.$$
\end{defn}
We have an analogous definition of $S \wwedge{j}(\mathfrak{p}^*) \subset \wwedge{j}(\mathfrak{p}^*)$. 

As a homomorphism from $\wwedge{kq}(\mathfrak{p})$ to $\mathrm{Pol}(V^k)$, the cocycle $\varphi_{kq}$ factors through $S\wwedge{j}(\mathfrak{p})$ (of $\wwedge{kq}(\mathfrak{p})$).  We will denote the space of such cocycles by \\
$SC^{kq}( \mathfrak{so}(p,q), SO(p) \times SO(q); \mathrm{Pol}(V^k))$, to be abbreviated $SC^{kq}$.

We then have the subbundle $S\wwedge{j}(T^*(D))$ given by
\begin{defn}
$$S\wwedge{j}(T^*(D)) = (G \times_ K S\wwedge{j}(\mathfrak{p}^*)).$$
\end{defn}

Since the subspace $S\wwedge{j}(\mathfrak{p})$ is invariant under the Riemannian holonomy group $K = SO(p) \times SO(q)$, the space of sections of the subbundle $S\wwedge{j}(T^*(D))$ is invariant under harmonic projection.  We will define the subspace $SH^{kq}(M)$ to be the subspace of those $\omega \in H^{kq}(M)$ such that the harmonic projection of some representative closed form is a section of the subbundle $S\wwedge{j}(T^*(D))$.

It is obvious that the cocycle $\varphi_q$ belongs to $SC^q$ and hence its outer exterior powers $\varphi_{kq}$ belong to $SC^{kq}$.  Hence the $kq$-forms  $\phi_{kq}(\mathbf{x}, Z, \tau)$ defined in Equation \eqref{specialforms} are sections of 
$S\wwedge{kq}(T^*(M))$ and hence their harmonic projections are also sections of
 $S\wwedge{kq}(T^*(M))$.
We obtain 
\begin{lem}
$$[\phi_{kq}(\mathbf{x}, Z, \tau)] \in SH^{kq}(M).$$
\end{lem} 

We can now describe the subspace of the cohomology of $H^{kq}(M)$ that can be obtained from the cocycle $\varphi_{kq}$.  The following theorem follows from Theorem 10.10 of \cite{BMM1}. For the description of the Euler class  $e_q$, see \cite{BMM1} page 6 and subsection 5.12.1.


\begin{thm}
The special subspace $SH^{kq}(M)$ of $H^{kq}(M)$ is generated by the products of the  classes $[\phi_{jq} (\mathbf{x},Z,\tau)]$, $1 \leq j \leq k$, with the powers of the Euler class $e_q$ as $x$ and $\tau$ vary. 

\end{thm}




\subsection{Further results}

In a paper in preparation with Yousheng Shi, we have shown that in case $k=1$, $U(\mathfrak{sl}(2, \R)) \varphi_q = H^q(\mathfrak{so}(p,q), \SO(p) \times \SO(q); \mathcal{P}_1)$.

Since the $E_0$ term of the spectral sequence developed in this paper is a Koszul complex (see Chapter \ref{spectralchapter}), we can use the existing computer program Macaulay 2 to compute the $E_1$ term of the spectral sequence and thereby prove vanishing theorems of the relative Lie algebra cohomology for small $p,q,k$.  For example, we have computed that $H^\ell(\mathfrak{so}(2,2), \SO(2) \times \SO(2); \mathcal{P}_1) \neq 0$ if and only if $\ell = 2, 3, \text{ or } 4$.  We have also shown that $H^\ell(\mathfrak{so}(2,2), \SO(2) \times \SO(2); \mathcal{P}_2) = 0$ if and only if $\ell \neq 4$.  We have shown $H^\ell(\mathfrak{so}(3,2), \SO(3) \times \SO(2); \mathcal{P}_1) = 0$ if $\ell = 0,1 \text{ or } 5$ and is nonzero if $\ell = 2, 3, \text{ or } 6$.  We do not know if the cohomology vanishes or not in case $\ell = 4$.  We have also shown $H^\ell(\mathfrak{so}(3,2), \SO(3) \times \SO(2); \mathcal{P}_2) = 0$ if $\ell < 4$ and is nonzero in case $\ell = 4, 6$.  We do not know if the cohomology vanishes in case $\ell = 5$.

%% file: Chapter2.tex

\renewcommand{\thechapter}{2}

\chapter{Preliminaries} \label{Introchapter}

\section{Relative Lie algebra cohomology}

\subsection{The relative Lie algebra complex $(C^\bullet(\mathfrak{g}, K;\mathcal{V}),d)$}

Given a semi-simple Lie group $G$ with Lie algebra $\mathfrak{g}$ and maximal compact $K$ with Lie algebra $\mathfrak{k}$, we have the following splitting, orthogonal for the Killing form
\begin{equation*}
\mathfrak{g} = \mathfrak{k} \oplus \mathfrak{p}.
\end{equation*}

Let $\{\omega_i\}, \{e_i\}$ be dual bases for $\mathfrak{p}^*, \mathfrak{p}$.
As in Borel and Wallach \cite{BW} (page 13), we define the relative Lie algebra complex $C^\bullet(\mathfrak{g}, K; \mathcal{V})$ for $\mathcal{V}$ a $(\mathfrak{g}, K)$-module with action $\rho$ by
\begin{defn}
\begin{equation}
C^\ell(\mathfrak{g}, K; \mathcal{V}) = \big( \wwedge{\ell}(\mathfrak{g}/\mathfrak{k})^* \otimes \mathcal{V} \big)^K.
\end{equation}
To be precise, $C^\ell(\mathfrak{g}, K; \mathcal{V})$ is those elements $\omega \in \wwedge{\ell}(\mathfrak{g}/\mathfrak{k})^* \otimes \mathcal{V}$ such that for $k \in K$,
\begin{equation*}
(\mathrm{Ad}k)^*(\omega) = \rho(k) \omega.
\end{equation*}
Also, for $\omega \in C^\ell(\mathfrak{g}, K; \mathcal{V})$, $d$ is given by
\begin{equation*}
d \omega (x_1, \ldots, x_{\ell+1}) = \sum_{j=1}^{\ell+1} (-1)^{j-1} \rho(x_j) \omega(x_1, \ldots, \widehat{x_j}, \ldots, x_{\ell+1}).
\end{equation*}

\end{defn}

Using the basis fixed above, we have the following formula for calculating $d$.

\begin{lem} \label{relatived}
For $\omega \in C^{\ell}(\mathfrak{g}, K; \mathcal{V})$ we have 
\begin{equation*} 
d \omega = \sum_i \big( A(\omega_i) \otimes \rho(e_i) \big) \omega
\end{equation*}
where $A(\omega_i)$ denotes the operation of left exterior multiplication by $\omega_i$.

\end{lem}



\begin{proof}
We show
$$\sum_{i=1}^N (A(\omega_i) \wedge \rho(e_i) \omega) (x_1, \ldots, x_{\ell+1}) = \sum_{j=1}^{\ell+1} (-1)^{j-1} \rho(x_j) \omega(x_1, \ldots, \widehat{x_j}, \ldots, x_{\ell+1})$$.
We have
\begin{align*}
\sum_{i=1}^N (A(\omega_i) \wedge \rho(e_i) \omega) &(x_1, \ldots, x_{\ell+1}) \\
&= \sum_{i=1}^N \sum_{j=1}^{\ell+1} (-1)^{j-1} \omega_i(x_j) \rho(e_i) \omega(x_1, \ldots, \widehat{x_j}, \ldots, x_{\ell+1}) \\
&= \sum_{j=1}^{\ell+1} (-1)^{j-1} \sum_{i=1}^N \omega_i(x_j) \rho(e_i) \omega(x_1, \ldots, \widehat{x_j}, \ldots, x_{\ell+1}) \\
&= \sum_{j=1}^{\ell+1} (-1)^{j-1} \rho(\sum_{i=1}^N \omega_i(x_j) e_i) \omega(x_1, \ldots, \widehat{x_j}, \ldots, x_{\ell+1}) \\
&= \sum_{j=1}^{\ell+1} (-1)^{j-1} \rho(x_j) \omega(x_1, \ldots, \widehat{x_j}, \ldots, x_{\ell+1})
\end{align*}
where the final equality follows since $\sum_{i=1}^N \omega_i(x_j) e_i = x_j$.

\end{proof}

\subsection{The connection with the de Rham complex $((A^\bullet(G/K), \mathcal{V})^G, d)$}

The following proposition, stated in \cite{BW} page 15, provides the motivation for the definition of relative Lie algebra cohomology.

\begin{prop} \label{derhamisomorphism}
Given the symmetric space $D = G/K$ and $\pi: G \to D$, the following map gives an isomorphism of the relative Lie algebra complex and the de Rham complex of $\mathcal{V}$-valued $G$-invariant forms  on the symmetric space $D$.
\begin{align*}
\big( A^{\ell}(D, \mathcal{V})^G, d \big) &\cong \big(C^{\ell}(\mathfrak{g}, K; \mathcal{V}),d\big) \\
\omega &\mapsto \pi^* \omega|_e.
\end{align*}
\end{prop}

\begin{proof}
Given $\omega \in \big( A^{\ell}(D, \mathcal{V})\big)^G$, we have
$$L_{g^{-1}}^* \circ \rho(g) \omega = \omega \text{ or } L_{g}^* \omega = \rho(g) \omega.$$
Since $\omega \in A^\ell(G/K, \mathcal{V})$, for $k \in K$,
$$R_k^* \pi^* \omega = \pi^* \omega.$$
That is,
\begin{equation} \label{1}
R_k^* \pi^* \omega = \pi^* \omega \text{ and }
\end{equation}
\begin{equation} \label{2}
L_{k}^* \pi^* \omega = \rho(k) \pi^* \omega.
\end{equation}
Hence
\begin{align*}
L_{k}^* R_{k^{-1}}^* \pi^* \omega &= \rho(k) \pi^* \omega \text{ and thus }\\
(\mathrm{Ad}k)^* \pi^* \omega &= \rho(k) \pi^* \omega \text{ and }\\
\mathrm{Ad}^*(k) \circ \rho(k) \pi^* \omega|_e &= (\mathrm{Ad}{k^{-1}})^* \rho(k) \pi^* \omega|_e = \pi^* \omega|_e.
\end{align*}

Thus $\pi^* \omega|_e \in C^{\ell}(\mathfrak{g}, K; \mathcal{V})$.
Now we show the map $\omega \mapsto \pi^* \omega|_e$ is a map of complexes.  That is, it preserves the differentials.  This will be proved in Lemma \ref{mainlemmahere}.


\end{proof}

First, note that we have the trivial bundle $\pi: G \times \mathcal{V} \rightarrow G$ equipped with the $G$-action
$$g_0 (g, v) = (g_0 g, g_0v) := (g_0 g, \rho(g_0)v).$$
Note also, if $v \in \mathcal{V}$, then the constant section
$$s_v(g) = (g, v)$$
is not $G$-invariant.  However, if we define a section $\tilde{s}_v$ by
$$\tilde{s}_v(g) = (g, gv)$$
we obtain a $G$-invariant section by
the following lemma. We leave the proof to the reader.

\begin{lem}
$\tilde{s}_v$ is $G$-invariant.  That is,
$$\rho(g_0) \tilde{s}_v(g_0^{-1} g) = \tilde{s}_v(g).$$
\end{lem}
Choose a basis $\mathcal{B} = \{v_i\}_{i \in I}$ for $\mathcal{V}$ (note $I$ could be uncountable) and let \\
$\omega \in A^\ell(G/K, \mathcal{V})^G$.  Then we have the following lemma.

\begin{lem} \label{finitesum}
There exists a finite subset of independent vectors $\{v_1, \ldots, v_n\} \subset \mathcal{V}$ such that
$$\pi^* \omega = \sum_{i=1}^n \omega_i \otimes \tilde{s}_{v_i}$$
where $\omega_i$ is left $G$-invariant.

\end{lem}

\begin{proof}
The image of $\wwedge{\ell}(\mathfrak{p})$ under $\pi^* \omega|_e$ is a finite dimensional subspace of $\mathcal{V}$.  Choose a basis for that subspace.  Then there exists elements $\alpha_1, \ldots, \alpha_n \in \wwedge{\ell}(\mathfrak{p}^*)$ such that
$$\pi^* \omega|_e = \sum_{i=1}^n \alpha_i \otimes v_i.$$
We see, by applying $L_{g^{-1}}^* \otimes \rho(g)$ to both sides of the above equation, that
$$\pi^* \omega|_g = \sum_{i=1}^n L_{g^{-1}}^*\alpha_i \otimes \tilde{s}_{v_i}.$$
Then $\omega_i$ defined by $\omega_i|_g = L_{g^{-1}}^* \alpha_i$ is left $G$-invariant by definition.

\end{proof}



Let $x_1, \ldots, x_N$ be a basis for $\mathfrak{p}$ and $X_1, \ldots, X_N$ be the corresponding left-invariant horizontal vector fields on $G$.  Then we have the following lemma.

\begin{lem} \label{mainlemmahere}
\begin{equation*}
\pi^*(d \omega|_e)(x_1, \ldots, x_{\ell +1}) = \sum_{j=1}^{\ell+1} (-1)^{j-1} \rho(x_j) \pi^* \omega|_e (x_1, \ldots, \widehat{x_j}, \ldots, x_{\ell +1})
\end{equation*}
\end{lem}

\begin{proof}
First, note that
$$\pi^* d \omega = d \pi^* \omega.$$
The definition of the exterior derivative of $\mathcal{V}$-valued forms on $G$ gives
\begin{align*}
d \pi^* \omega(X_1, \ldots, X_{\ell +1}) &= \sum_j (-1)^{j-1} X_j \pi^* \omega (X_1, \ldots, \widehat{X_j}, \ldots, X_{\ell +1}) \\
&+ \sum_{j<k} (-1)^{j+k-1} \pi^* \omega([X_j, X_k], X_1, \ldots, \widehat{X_j}, \ldots, \widehat{X_k}, \ldots, X_{\ell+1}).
\end{align*}
The second term is zero, however, since $[\mathfrak{p}, \mathfrak{p}] \subset \mathfrak{k}$ and $\pi^* \omega$ is basic.  We claim we have
\begin{equation*}
X_j (\pi^* \omega|_e(X_1, \ldots, \widehat{X_j}, \ldots, X_{\ell+1}))|_e = \rho(x_j) \pi^* \omega|_e(x_1, \ldots, \widehat{x_j}, \ldots, x_{\ell+1}).
\end{equation*}
Indeed, by Lemma \ref{finitesum},
\begin{align*}
X_j (\pi^* \omega|_e(X_1, \ldots, \widehat{X_j}, \ldots, X_{\ell+1})) &= X_j (\sum_i \omega_i(X_1, \ldots, \widehat{X_j}, \ldots, X_{\ell+1}) \otimes \tilde{s}_{v_i}) \\
&= \sum_i X_j (\omega_i(X_1, \ldots, \widehat{X_j}, \ldots, X_{\ell+1})) \otimes \tilde{s}_{v_i} \\
&+ \sum_i \omega_i(X_1, \ldots, \widehat{X_j}, \ldots, X_{\ell+1}) \otimes X_j(\tilde{s}_{v_i}).
\end{align*}
The first term is zero because $\omega_i(X_1, \ldots, \widehat{X_j}, \ldots, X_{\ell+1})$ is constant (the $\omega_i$'s and the $X_j$'s are left-invariant).  Also,
\begin{align*}
X_j \tilde{s}_{v_i}(g) &= \frac{d}{dt}|_{t=0} \tilde{s}_{v_i}(g e^{tX_j}) \\
&= (g, \frac{d}{dt}|_{t=0} g e^{tX_j} e_i) \\
&= (g, g \rho(X_j) e_i).
\end{align*}
Hence, evaluating at the identity we have
\begin{equation*}
X_j \tilde{s}_{v_i} (e) = (e, \rho(x_j) e_i).
\end{equation*}

Thus,
\begin{equation*}
\pi^* d\omega|_e = d \pi^* \omega|_e = \sum_j (-1)^{j-1} \rho(x_j) \omega|_e(x_1, \ldots, \widehat{x_j}, \ldots, x_{\ell+1}).
\end{equation*}


\end{proof}

\section{The Weil representation} \label{weilrepsection}

Let $(V, (\ ,\ ))$ be an orthogonal space of signature $p,q$ and $(W, <\ ,\ >)$ be a real vector space of dimension $2k$ equipped with a non-degenerate skew-symmetric form.  Then we can form the space $V \otimes W$ with non-degenerate skew-symmetric form $(\ ,\ ) \otimes <\ ,\ >$.  We will construct the Fock model, $\mathcal{P}_k$, of the Weil representation $(\varpi, \mathfrak{sp}(V \otimes W))$ and give formulas for how $\mathfrak{o}(V)$ operates in this model.

\begin{rmk}
We will adopt the convention of using ``early" greek letters $\alpha, \beta$ to denote integers between $1$ and $p$ and ``late" greek letters $\mu, \nu$ to denote integers between $p+1$ and $p+q$.
\end{rmk}

Let $e_1,\ldots,e_{p+q}$ be an orthogonal basis for $V$ such that $(e_\alpha ,e_\alpha) =1, 1 \leq \alpha \leq p$ and $(e_\mu,e_\mu) = -1, p+1 \leq \mu \leq p+q$.  Let $V_+$ be the span of $e_1, \ldots, e_p$ and $V_-$ be the span of $e_{p+1}, \ldots, e_{p+q}$.  Then we have the splitting $V = V_+ \oplus V_-$.  We have the splitting (orthogonal for the Killing form)
$$\mathfrak{so}(p,q) = \mathfrak{k} \oplus \mathfrak{p}_0$$
and denote the complexification $\mathfrak{p}_0 \otimes \C$ by $\mathfrak{p}$.

We recall that the map $\phi:\bigwedge^2(V) \to \mathfrak{so}(p,q)$ given by
$$\phi(u \wedge v)(w) = (u,w)v - (v,w)u $$
is an isomorphism.  Under this isomorphism the elements $e_\alpha \wedge e_\mu, 1 \leq \alpha \leq p, p+1 \leq \mu \leq p+q$, are a basis for $\mathfrak{p}_0$ and the elements $e_\alpha \wedge e_\beta$ and $e_\mu \wedge e_\nu$ are a basis for $\mathfrak{k}$, where $, 1 \leq \alpha, \beta \leq p, p+1 \leq \mu, \nu \leq p+q$.  We define $e_{\alpha, \mu} = -e_\alpha \wedge e_\mu$ and let $\{ \omega_{\alpha \mu} \}$ be the dual basis for $\mathfrak{p}_0^*$.  We  define $\vol \in (\wwedge{pq}\mathfrak{p}_0^*)^K$ by
\begin{equation} \label{voldefn}
\vol = \omega_{1, p+1} \wedge \cdots \wedge \omega_{p, p+q}.
\end{equation}

Now let $(V \otimes \C)^k = \displaystyle \bigoplus_{i=1}^k (V \otimes \C)$.  We will use $\mathbf{v}$ to denote the element $(v_1,v_2,\cdots,v_k) \in (V \otimes \C)^k$.  We will often identify $(V \otimes \C)^k$ with the $((p+q) \times k)$-matrices 
$$\begin{pmatrix}
z_{11} & z_{12} & \cdots & z_{1k} \\
\vdots & \vdots & \ddots & \vdots \\
z_{p+q, 1} & z_{p+q, 2} & \cdots & z_{p+q,k}
\end{pmatrix}$$
over $\C$ using the basis $e_1, \ldots, e_{p+q}$.  Then $\mathbf{v}$ will correspond to the $(p+q) \times k$ matrix $Z$ where $v_j$ is the $j^{th}$ column of the matrix.  The splitting $V = V_+ \oplus V_-$ induces the splitting $(V \otimes \C)^k = (V_+ \otimes \C)^k \oplus (V_- \otimes \C)^k$. 

We define $\mathcal{P}_k$ to be the space of holomorphic polynomials on $(V \otimes \C)^k$.  That is,
\begin{equation*}
\mathcal{P}_k = \mathrm{Pol}((V \otimes \C)^k).
\end{equation*}

By Theorem 7.1 of \cite{KM2}, we have the following formulas for the action of $\mathfrak{o}(V)$ on $\mathcal{P}_k$ (note that we have set their parameter $\lambda$ equal to $\frac{1}{2i}$).  Note that in the reference there is an overall sign error.

\begin{prop} \label{actionofov}
\begin{align*}
\varpi(e_\alpha \wedge e_\beta) &= -\sum_{i=1}^k (z_{\alpha i} \frac{\partial}{\partial z_{\beta i}} - z_{\beta i} \frac{\partial}{\partial z_{\alpha i}}) \\
\varpi(e_\mu \wedge e_\nu) &= \sum_{i=1}^k (z_{\mu i} \frac{\partial}{\partial z_{\nu i}} - z_{\nu i} \frac{\partial}{\partial z_{\mu i}}) \\
\varpi(e_\alpha \wedge e_\mu) &= \sum_{i=1}^k (-z_{\alpha i} z_{\nu i} + \frac{\partial^2}{\partial z_{\alpha i}\partial z_{\mu i}}).
\end{align*}
\end{prop}

\begin{rmk}
We see from the above proposition that $\mathfrak{k}$ acts diagonally in the ``usual" way, perhaps twisted by some character, hence we may twist the representation by a power of $\det$ so that $K = \SO(p) \times \SO(q)$ acts in the usual way.  That is, given $f \in \mathcal{P}_k$ and $g \in K$,
\begin{equation*}
\varpi(g)f(\mathbf{v}) = f(g^{-1} \mathbf{v}).
\end{equation*}
The unexpected action of $\mathfrak{p}$ is due to the model we have chosen.
\end{rmk}

We will be concerned with the complex $\big( C^\ell(\mathfrak{so}(p,q), \SO(p) \times \SO(q); \mathcal{P}_k),d \big)$.  By Proposition \ref{relatived} and Proposition \ref{actionofov}, we have the following formula for $d$,

\begin{equation} \label{defnofd}
d = \sum_{i=1}^k \sum_{\alpha=1}^p \sum_{\mu = p+1}^{p+q} A(\omega_{\alpha \mu}) \otimes \big( \frac{\partial^2}{\partial z_{\alpha i} \partial z_{\mu i}} - z_{\alpha i} z_{\mu i} \big).
\end{equation}
We note that
\begin{equation*}
C^{\bullet} \big(\mathfrak{so}(p,q) ,\SO(p) \times \SO(q); \mathcal{P}_k \big) \cong C^{\bullet} \big(\mathfrak{so}(p,q, \C) ,\SO(p, \C) \times \SO(q, \C); \mathcal{P}_k \big).
\end{equation*}
We will use this isomorphism between cochain complexes throughout the paper.

%% file: Chapter3.tex

\renewcommand{\thechapter}{3}

\chapter{The Spectral Sequence Associated to the Relative Lie Algebra Cohomology of the Weil Representation} \label{spectralchapter}

\section{The spectral sequence associated to a filtered complex} \label{spectralsection}
In this section we will construct a  spectral sequence associated to a filtered complex.  Our basic reference will be Chapter 2 of \cite{Mc}, especially Theorem 2.1.  However, we warn the reader that the convergence part of \cite{Mc} Theorem 2.1 will not apply to our case, since our filtration is not assumed to be bounded below.

In what follows we will assume we have a filtered cochain complex $(F^\bullet, C, d)$ with (cohomological) degrees between $0$ and $n$ for some fixed $n$, that is,
$$d(F^p C^\ell) \subset F^p C^{\ell+1}.$$
In fact, we will assume this filtration is {\it decreasing}, that is,
$$F^{p+1} C \subset F^p C.$$

\subsection{Some general results on spectral sequences}
We first recall that a spectral sequence is a sequence of bigraded (by $\Z \times \Z$) complexes $\{ E^{\bullet,\bullet}_r\, d_r\}$ such that
\begin{equation} \label{definingequationforE}
H^{p,q}(E_r,d_r) \cong E^{p,q}_{r+1}, \ p,q \in \Z \times \Z, r \geq 0.
\end{equation}
\begin{rmk}
Note that $d_r$ is bigraded, hence it will have a bidegree $(a,b)$.
\end{rmk}

We recall the following definitions.
\begin{defn}
A filtration $F^\bullet$ of a cochain complex $C$ is {\it exhaustive} if
$$\bigcup_{p \in \Z} F^p C = C $$
and {\it separated} if
$$\bigcap_{p \in \Z} F^p C = 0.$$
\end{defn}

Many occurences of spectral sequences come from the following theorem, see \cite{Mc} Theorem 2.1. Note that the defining formulas for $Z^{p,q}_r$  and $B^{p,q}_r$ on page 33 of \cite{Mc} are not correct as stated but the correct formulas are used throughout pages 33-35, in particular in the proof of Theorem 2.1.  We define them below for clarity.

Define subspaces $Z^{p,q}_r$ and $B ^{p,q}_r$, for $r \geq 1$, of $C^{p+q}$ by
\begin{enumerate}
\item $Z^{p,q}_r = \mathrm{ker}\big(d: F^pC^{p+q} \to  F^p C^{p+q+1}/F^{p+r}C^{p+q+1} \big) $
\item $B ^{p,q}_r = \mathrm{im}\big(d:F^{p-r+1}C^{p+q-1} \to F^pC^{p+q}\big)$.
\end{enumerate}
Thus, $Z_r^{p,q}$ are the elements $z \in F^p C^{p+q}$ such that $dz \in F^{p+r} C^{p+q+1}$ and $B_r^{p,q}$ are the elements $b \in F^p C^{p+q}$ such that there exists $x \in F^{p-r+1} C^{p+q}$ with $dx = b$.  Note that elements of $B_r^{p,q}$ are ``absolute" coboundaries whereas elements of $Z_r^{p,q}$ need not be ``absolute" cocycles.  That is, we have
$$ Z^{p,q}_1 \supset Z^{p,q}_2\supset \cdots  \supset Z^{p,q} \supset B^{p,q} \supset \cdots \supset B^{p,q}_2 \supset B^{p,q}_1 .$$
For the sake of consistent notation, we define the following subspaces of $C^{p,q}$.
\begin{equation*}
\begin{aligned}
Z_{-1}^{p,q} &= Z_0^{p,q} = F^p C^{p+q} \\
B_0^{p,q} &= 0.
\end{aligned}
\end{equation*}

We note two properties of $\{Z^{p,q}_r\}$ and $\{B^{p,q}_r\}$. 
\begin{lem} \label{exhaustsep}
We have
\begin{enumerate}
\item If $F^\bullet$ is exhaustive then $B^{p,q} = \bigcup_{r} B^{p,q}_r$.
\item If $F^\bullet$ is separated then $Z^{p,q} = \bigcap_{r} Z^{p,q}_r$.
\end{enumerate} 
\end{lem}

We now define a candidate for the $r^{th}$ bigraded complex of the spectral sequence associated to a filtration.  Define the quotient space $E_r^{p,q}$ by
\begin{equation} \label{Erpqdef}
E^{p,q}_r = \frac{Z^{p,q}_r}{B^{p,q}_r  + Z^{p+1,q-1}_{r-1}}, r \geq 0
\end{equation}
and let
$$E_r = \bigoplus_{p,q} E_r^{p,q}.$$
Note that, by definition,
$$E_0^{p,q} = \frac{F^p C^{p,q}}{F^{p+1} C^{p+q}}.$$
Together with the action $d_0$ (the differential induced on the quotient by $d$), $E_0$ is then the graded complex associated to the filtration $F^\bullet$.  We will often refer to this bigraded complex as the associated graded complex.

\begin{thm}\label{existenceofspecseq}
Suppose $(F^\bullet, C,d)$ is a cochain complex equipped with a decreasing filtration $F^{\bullet} C$.  Then there is a spectral sequence $\{ E^{\bullet,\bullet}_r\, d_r\}$ with first term $E_0$, the associated graded complex, and $E_r$ defined in Equation \eqref{Erpqdef}.  The differential $d_r$ is the differential $d$ of $C$ restricted to $Z_r^{p,q} \subset C^{p+q}$ for $r \geq 0$.  Note that $d_r$ has bidegree $(r, -r+1)$ for $r \geq 0$.
\end{thm}

\begin{proof}

We claim $d$ induces a map $d_r$ on the $r^{th}$ page,
\begin{align*}
d_r: E_r^{p,q} &\rightarrow E_r^{p+r,q-r+1} \\
\frac{Z^{p,q}_r}{B^{p,q}_r  + Z^{p+1,q-1}_{r-1}} &\rightarrow \frac{Z^{p+r,q-r+1}_r}{B^{p+r,q-r+1}_r  + Z^{p+r+1,q-r}_{r-1}}.
\end{align*}

Let $z \in Z_r^{p,q}$.  Then $dz$ is closed and in $F^{p+r}C^{p+q+1}$, so $dz \in Z^{p+r,q-r+1} \subset Z_r^{p+r,q-r+1}$.  So $d$ induces a map $Z_r^{p,q} \rightarrow Z^{p+r,q-r+1}_r$ and hence a map to the quotient $E_r^{p+r,q-r+1}$.  Now we show it factors through the quotient $E_r^{p,q}$.

We must show $d(B^{p,q}_r  + Z^{p+1,q-1}_{r-1}) \subset (B^{p+r,q-r+1}_r  + Z^{p+r+1,q-r}_{r-1})$.  But $d(B^{p,q}_r) = 0$, so all that remains is to show $d(Z^{p+1,q-1}_{r-1}) \subset (B^{p+r,q-r+1}_r  + Z^{p+r+1,q-r}_{r-1})$.  Let $z \in Z^{p+1,q-1}_{r-1}$.  Then $dz$ is a boundary and in $F^{p+r} C^{p+q+1}$, thus $dz \in B_r^{p+r, q-r+1}$.

Finally, note that since $d^2 = 0$ we have $d_r^2 = 0$.  Now we show that the cohomology of the $r^{th}$ page is the $(r+1)^{st}$ page.  Accordingly, we define the kernel and image of the map $d_r$.  We define the graded subspaces $\overline{Z}^{p,q}_r$ and $\overline{B}^{p,q}_r$ of $E^{p,q}_{r-1}$ by
\begin{equation}
\begin{aligned}
\overline{Z}^{p,q}_r &= \mathrm{ker}\big(d_{r-1}:E^{p,q}_{r-1} \to E^{p+r-1,q-r+2}_{r-1} \bigskip),\\
\overline{B}^{p,q}_r &=\mathrm{im}\big(d_{r-1}:E^{p-r+1,q+r-2}_{r-1} \to E^{p,q}_{r-1} \big).
\end{aligned}
\end{equation}
Hence, by definition we have
\begin{equation*}
H^{p,q}(E^{p,q}_{r-1}) = \overline{Z}^{p,q}_r / \overline{B}^{p,q}_r .
\end{equation*}


We construct an isomorphism
$$H^{p,q}(E^{p,q}_{r-1}) \cong E_r^{p,q}.$$
First, consider $\overline{Z}_r^{p,q} = \mathrm{ker}(E_{r-1}^{p,q} \rightarrow E_{r-1}^{p+r-1,q-r+2})$.  This is, by definition,

\begin{equation*}
\mathrm{ker} \bigg(d: \frac{Z_{r-1}^{p,q}}{B_{r-1}^{p,q} + Z_{r-2}^{p+1,q-1}} \rightarrow \frac{Z_{r-1}^{p+r-1,q-r+2}}{B_{r-1}^{p+r-1,q-r+2} + Z_{r-2}^{p+r,q-r+1}} \bigg).
\end{equation*}

Suppose $z$ is in the kernel.  If $dz \in Z_{r-2}^{p+r,q-r+1}$, then $dz \in F^{p+r}C^{p+q+1}$ and $z \in F^pC^{p+q+1}$, hence $z \in Z_r^{p,q}$.  If $dz \in B_{r-1}^{p+r-1,q-r+2}$ then $z \in F^{p+1}C^{p+q}$ and $dz \in F^{p+r-1}C^{p+q+1}$ and hence $z \in Z_{r-2}^{p+1, q-1}$.  Thus

\begin{equation*}
\overline{Z}_r^{p,q} = \frac{Z_{r}^{p,q} + Z_{r-2}^{p+1,q-1}} {B_{r-1}^{p,q} + Z_{r-2}^{p+1,q-1}}.
\end{equation*}

Now consider $\overline{B}_r^{p,q} = \mathrm{im}(E_{r-1}^{p-r+1,q+r-2} \rightarrow E_{r-1}^{p,q})$.  By definition, this is
\begin{equation*}
\mathrm{im} \bigg(d: \frac{Z_{r-1}^{p-r+1,q+r-2}}{B_{r-1}^{p-r+1,q+r-2} + Z_{r-2}^{p-r+2,q+r-3}} \rightarrow \frac{Z_{r-1}^{p,q}}{B_{r-1}^{p,q} + Z_{r-2}^{p+1,q-1}} \bigg).
\end{equation*}
If $z \in Z_{r-1}^{p-r+1,q+r-2}$, then $dz \in B_{r}^{p,q}$.  Now suppose $w \in B_r^{p,q}$.  Then there is some $z^\prime \in C^{p-r+1,q+r-2}$ so that $dz^\prime = w$ and hence $z^\prime \in Z_{r-1}^{p-r+1,q+r-2}$.  Thus, since $B_{r-1}^{p,q} \subset B_r^{p,q}$
\begin{equation*}
\overline{B}_r^{p,q} = \frac{B_{r}^{p,q} + B_{r-1}^{p,q} + Z_{r-2}^{p+1,q-1}}{B_{r-1}^{p,q} + Z_{r-2}^{p+1,q-1}} = \frac{B_{r}^{p,q} + Z_{r-2}^{p+1,q-1}}{B_{r-1}^{p,q} + Z_{r-2}^{p+1,q-1}}.
\end{equation*}
Thus
\begin{equation} \label{sumtopsumbot}
\frac{\overline{Z}_r^{p,q}}{\overline{B}_r^{p,q}} = \frac{Z_{r}^{p,q} + Z_{r-2}^{p+1,q-1}}{B_{r}^{p,q} + Z_{r-2}^{p+1,q-1}}.
\end{equation}

Recall the following consequence of the Butterfly Lemma, $\frac{X+Y}{Y} \cong \frac{X}{X \cap Y}$.  Dividing the top and bottom of Equation \eqref{sumtopsumbot} by $Z_{r-2}^{p+1,q-1}$, we have
\begin{equation} \label{quottopquotbot}
\frac{\overline{Z}_r^{p,q}}{\overline{B}_r^{p,q}} = \frac{Z_{r}^{p,q} / (Z_{r}^{p,q} \cap Z_{r-2}^{p+1,q-1})}{B_{r}^{p,q} / (B_{r}^{p,q} \cap Z_{r-2}^{p+1,q-1})}.
\end{equation}

Simply checking the definitions yields $Z_{r}^{p,q} \cap Z_{r-2}^{p+1,q-1} = Z_{r-1}^{p+1,q-1}$ and $B_{r}^{p,q} \cap Z_{r-2}^{p+1,q-1} = B_{r+1}^{p+1,q-1}$.  Hence
\begin{equation*}
\frac{\overline{Z}_r^{p,q}}{\overline{B}_r^{p,q}} = \frac{Z_r^{p,q}/Z_{r-1}^{p+1,q-1}}{B_r^{p,q}/B_{r+1}^{p+1,q-1}}.
\end{equation*}

Another consequence of the Butterfly Lemma is the isomorphism $\frac{A}{B+C} \cong \frac{A/B}{C/(B\cap C)}$.  Finally, since $B_{r+1}^{p+1,q-1} = B_r^{p,q} \cap Z_{r-1}^{p+1,q-1}$, we have
\begin{equation*}
\frac{\overline{Z}_r^{p,q}}{\overline{B}_r^{p,q}} \cong \frac{Z_r^{p,q}}{B_r^{p,q} + Z_{r-1}^{p+1,q-1}} = E_r^{p,q}.
\end{equation*}

\end{proof}

We now give conditions on the filtration of a filtered complex that are sufficient to imply convergence of the associated spectral sequence.  The applications of this spectral sequence later in the paper will satisfy the hypotheses of Proposition \ref{Steveprop1}.

First, a simple lemma.

\begin{lem} \label{r(p)}
Suppose $(F^\bullet, C, d)$ is a filtered cochain complex such that the filtration is decreasing and bounded above.  Then the filtration is separated and for all $(p,q)$ there exists an integer $r(p)$ so that for all $r \geq r(p)$,
$$Z_r^{p,q} = Z^{p,q}.$$
\end{lem}

\begin{proof}
Suppose the filtration is bounded above by $P$.  That is, for all $p > P$ we have $F^p C^{p,q} = 0$.  Thus $F$ is separated.

Fix $(p,q)$.  The image of $E_r^{p,q}$ under $d_r$ is contained in $E_r^{p+r, q-r+1}$ for all $r,p,q$.  So, if $r > P-p$, then $p+r > P$ and $E_r^{p+r, q-r+1} = 0$.  Thus, elements which are pushed up enough in the filtration will necessarily be zero.  Taking $P-p = r(p)$ is sufficient.

\end{proof}

\begin{rmk}
The above lemma can be interpreted as follows.  Since the support of $C^{p,q}$ is bounded on the right in the $p,q$-plane, eventually all maps out of a given point on the page will have codomain outside the support of the bigraded complex, see for instance, the figure after Remark \ref{d=d_{0,1}}.
\end{rmk}

\begin{prop} \label{Steveprop1}
Suppose $(F^\bullet, C, d)$ is a filtered cochain complex such that the filtration is decreasing, bounded above, and exhaustive.  Then the spectral sequence converges.  That is, for all $p,q$, there is an isomorphism
\begin{equation}
E^{p,q}_\infty \rightarrow gr^{p,q}(H^\bullet).
\end{equation}
\end{prop}

\begin{proof}
Let $p$ be given. Because the filtration is bounded above, by Lemma \ref{r(p)} we have that for each $(p,q)$ there exists an $r(p)$ so that $Z_r^{p,q} = Z^{p,q}$ for all $r \geq r(p)$.  For $r > \mathrm{max}(r(p), r(p+1))$ we will construct below a surjective map 
\begin{equation} \label{pirmap}
\pi_{r} : E_r^{p,q} \cong \frac{Z_r^{p,q}}{B_r^{p,q} + Z_{r-1}^{p+1, q-1}} \to gr^{p,q}(H^\bullet) \cong \frac{Z^{p,q}}{B^{p,q} + Z^{p+1, q-1}}.
\end{equation}
Since $r > r(p)$, $Z^{p,q}_r =Z^{p,q}$ and we may first define $\widetilde{\pi}_r: Z^{p,q}_r \to Z^{p,q}$ to be the identity map.  We then define $\pi'_r$ to be the induced quotient map
$$\pi'_r:   Z^{p,q}_r  \to \frac{Z^{p,q}}{B^{p,q} + Z^{p+1, q-1}}.$$
Since $r > r(p+1)$, we have $Z^{p+1,q-1}_{r-1} =Z^{p+1,q-1}$.  Moreover, $B_r^{p,q} \subset B^{p,q}$. Hence the map $\pi'_r$ factors through the quotient by $B_r^{p,q} + Z_{r-1}^{p+1, q-1}$ to give the required surjection $\pi_r$. 

Note that there is a quotient map from $E_r^{p,q}$ to $E_{r+1}^{p,q}$ making $\{E_r^{p,q}, r> r(p)\}$ into a direct system. Moreover, the maps $\{ \pi_r: r > r(p)\}$ fit together to induce a morphism from the direct system to  $gr^{p,q}(H^\bullet)$ and hence we obtain a surjective map $\pi_{\infty}: E_{\infty} ^{p,q} \to gr^{p,q}(H^\bullet)$. We claim that $\pi_{\infty}$ is injective. Indeed suppose $x \in  E_{\infty} ^{p,q}$ satisfies $\pi_{\infty}(x) =0$.  Then for some $r$ we have $\pi_r(x) = 0$.  Hence $x \in B^{p,q} + Z^{p+1, q-1}$.  By Lemma \ref{exhaustsep}, we have $x \in B_{r'}^{p,q} + Z^{p+1, q-1}$ for some $r' \geq r$.  Furthermore, since $Z^{p+1,q-1} \subset Z_{r'}^{p+1,q-1}$ we have $x \in B_{r'}^{p,q} +Z_{r'}^{p+1,q-1}$.  Thus $x$ is zero in $E_{r'}^{p,q}$ and hence is zero in the direct limit.

\end{proof}

\begin{rmk}
The above proof highlights a fundamental concept of spectral sequences.  That is, $E_{r+1}^{p,q}$ is a subquotient of $E_r^{p,q}$.  It is the closed elements (``sub") quotiented by the exact elements.  With the hypotheses of Proposition \ref{Steveprop1} (or just those of Lemma \ref{r(p)}), eventually everything is closed and hence the $Z_r^{p,q}$ stabilize.  At this point, $E_{r+1}^{p,q}$ is just a quotient (rather than a subquotient) of $E_r^{p,q}$ and we obtain maps between the pages.
\end{rmk}

The spectral sequence above converges to the graded vector space associated to the induced filtration of the cohomology. We conclude the general discussion by describing this bigraded vector space.

\subsubsection{The associated graded $\mathrm{gr}(H^\bullet)$}
The filtration on $C$ induces filtrations on the cocycles $Z$ and coboundaries $B$.  Hence, it induces a filtration on the cohomology $H$.  A priori, the vector space $\mathrm{gr}^{p,q}(H)$ is a four-fold quotient, but we have the following proposition.

\begin{prop} \label{grH}
\begin{equation*}
\mathrm{gr}^{p,q}(H) \cong \frac{Z^{p,q}}{B^{p,q} + Z^{p+1,q-1}}.
\end{equation*}
\end{prop}

\begin{proof}
By definition,
\begin{equation*}
\mathrm{gr}^{p,q}(H) = \frac{Z^{p,q}/B^{p,q}}{Z^{p+1,q-1}/B^{p+1,q-1}}.
\end{equation*}
By one of the standard isomorphism theorems,
\begin{equation*}
\frac{Z^{p,q}}{B^{p,q} + Z^{p+1,q-1}} \cong \frac{Z^{p,q}/B^{p,q}}{(B^{p,q} + Z^{p+1,q-1})/B^{p,q}}.
\end{equation*}
By another standard isomorphism theorem,
\begin{equation*}
\frac{B^{p,q} + Z^{p+1,q-1}}{B^{p,q}} \cong \frac{Z^{p+1,q-1}}{B^{p,q} \cap Z^{p+1,q-1}}.
\end{equation*}
Finally, observe that $B^{p,q} \cap Z^{p+1,q-1} = B^{p+1,q-1}$.

\end{proof}

\subsection{Some consequences of the vanishing of $E_1^{p,q}$} \label{consequencesofE_1vanishsection}

In the spectral sequences which follow, many of the terms $E_1^{p,q}$ will vanish.  To utilize this feature, we need the following two general propositions from the theory of spectral sequences.  In what follows we assume the filtration $F^{\bullet}$ is bounded above and exhaustive.

The following proposition is an immediate consequence of convergence of the spectral sequence (Proposition \ref{Steveprop1}) since $\mathrm{gr}(H)$ is obtained from $E_1$ by taking successive subquotients.

\begin{prop}\label{grCzeroimpliesCzero}
Suppose $(F^\bullet,C, d)$ is a filtered cochain complex such that $F^\bullet$ is bounded above and exhaustive.  Then $H^\ell(\mathrm{gr}(C))= 0$ implies $H^\ell(C) = 0$.
\end{prop} 

\begin{rmk}
In the case we are studying, $C^\bullet = C^\bullet(\mathfrak{so}(p,q), \SO(p) \times \SO(q); \mathcal{P}_k)$ has a canonical grading as a vector space.  Hence, in this case, there is a map of graded vector spaces $f: C \to \mathrm{gr}(C)$ which sends $\varphi \in C$ to its leading term.  However, $f$ does not commute with the differential. Hence, there is no map in general (even if $C$ is graded as a vector space)  from the cohomology of $C$ to the cohomology of $\mathrm{gr} (C) $.
\end{rmk} 

\begin{prop} \label{generalspectral}
Let $(F^\bullet,C, d)$ be a filtered cochain complex such that $F^\bullet$ is bounded above and exhaustive.
\begin{enumerate}
\item If $H^{\ell-1}(\mathrm{gr}(C)) = 0$, then there is a well defined map from $H^\ell(C)$ to $H^\ell(\mathrm{gr}(C))$ and it is an injection.
\item If $H^{\ell+1}(\mathrm{gr}(C)) = 0$, then there is a well defined map from $H^\ell(\mathrm{gr}(C))$ to $H^\ell(C)$ and it is a surjection.
\item If $H^{\ell-1}(\mathrm{gr}(C)) = 0$ and $H^{\ell+1}(\mathrm{gr}(C)) = 0$, then the map from $H^\ell(C)$ \\
to $H^\ell(\mathrm{gr}(C))$ is an isomorphism.
\end{enumerate}

\end{prop}

\begin{proof}
We first prove (1). We will construct an inverse system of injective maps
$$\cdots \hookrightarrow E_r^{p,q} \hookrightarrow E_{r-1}^{p,q} \hookrightarrow E_{r-2}^{p,q} \hookrightarrow \cdots \hookrightarrow E_2^{p,q} \hookrightarrow E_1^{p,q}.$$
First note by the hypothesis of (1) we have
\begin{equation} \label{coboundariesvanish}
\overline{B}^{p,q}_r = 0, r \geq 2.
\end{equation}
Next, note that for any spectral sequence $\{E_r, d_r\}$ we have an inclusion
\begin{equation} \label{inclusioncobound}
\overline{Z}^{p,q}_r \hookrightarrow E_{r-1}^{p,q}, r \geq 1.
\end{equation}
But since $E^{p,q}_r = \overline{Z}^{p,q}_r / \overline{B}^{p,q}_r$, by \eqref{coboundariesvanish} we have 
$$E_{r}^{p,q} = \overline{Z}^{p,q}_r, r \geq 2$$
and the inclusion of Equation \eqref{inclusioncobound} becomes 
$$E_{r}^{p,q}\hookrightarrow E_{r-1}^{p,q}, r \geq 2.$$
Thus $\{E_r^{p,q} \}$ is an inverse system of injections which may be identified with a decreasing (for inclusion) sequence of bigraded subspaces of the fixed bigraded vector space $E_1$.  We have constructed the required inverse system. 

The inverse limit of the above sequence is $E^{p,q}_{\infty}$.  Since the inverse limit of an inverse system maps to each member of the system, we have a map $E_\infty^{p,q} \rightarrow E_1^{p,q}$.  In this case the inverse limit is simply the intersection of all the subspaces and the map of the limit is the inclusion of the infinite intersection which is obviously an injection.  Since we have convergence (Proposition \ref{Steveprop1}), $E_\infty^{p,q} \cong \mathrm{gr}^{p,q}(H)$ and $E_1^{p,q} = H^{p,q}(\mathrm{gr}(C))$.  Hence (1) is proved.


We now prove (2).  We construct a direct system of surjective maps
$$E_1^{p,q} \twoheadrightarrow E_2^{p,q} \twoheadrightarrow \cdots \twoheadrightarrow  E_{r-1}^{p,q} \twoheadrightarrow E_r^{p,q} \twoheadrightarrow E_{r+1}^{p,q} \twoheadrightarrow  \cdots.$$ 

First note by the hypothesis of (2) we have $d_{r-1}|_{E_{r-1}^{p,q}} = 0, r \geq 2, p+q = \ell$ and hence
\begin{equation} \label{everythingisacycle}
E^{p,q}_{r-1}  = \overline{Z}^{p,q}_r = 0, r \geq 2.
\end{equation}
Next, note that for any spectral sequence $\{E_r, d_r\}$ we have a surjection
$$\overline{Z}^{p,q}_r \twoheadrightarrow E_{r}^{p,q}, r \geq 1.$$
Hence, by Equation \eqref{everythingisacycle}, the previous surjection becomes 
$$E_{r-1}^{p,q}\twoheadrightarrow E_{r}^{p,q}, r \geq 2.$$
Hence $\{E_r^{p,q} \}$ is a direct system of surjections which may be identified with a  sequence of bigraded quotient spaces of the fixed bigraded vector space $E_1$. We have constructed the required direct system. 

Since each member of a direct system maps to the direct limit, the space $E_1^{p,q}$ maps to $E_\infty^{p,q}$ and this map is clearly surjective.  As in the proof of (1), since we have convergence (Proposition \ref{Steveprop1}), $E_\infty^{p,q} \cong \mathrm{gr}^{p,q}(H)$ and $E_1^{p,q} = H^{p,q}(\mathrm{gr}(C))$.  Hence (2) is proved.

Lastly, (3) is obvious.

\end{proof}

\section{Construction of the spectral sequence for the relative Lie algebra cohomology of the Weil representation.} \label{constructionofss}

We now study the above spectral sequence for the case in hand, $G = \SO(a,b)$, and apply the previous results to it.  We let $V$ be a real vector space of signature $a,b$.  Recall that we use $\mathcal{P}_k$ to denote the ring $\mathrm{Pol}((V \otimes \C)^k) $.  This ring is graded by polynomial degree
$$\mathcal{P}_k = \bigoplus_{i=0}^\infty \mathcal{P}_k(i).$$
This grading of $\mathcal{P}_k$ induces a grading of $C^\ell = C^\ell(\mathfrak{so}(a,b), \SO(a) \times \SO(b); \mathcal{P}_k)$, for each $\ell$, called the polynomial grading. We will let $C^\ell(i)$ denote the $i$-th graded summand of $C^\ell$.  The above grading of $C^\ell$ induces an increasing filtration $F_{\bullet}$  of $C^\ell$ by
$$F_p C^\ell= \bigoplus_{i=0}^p C^\ell(i).$$
The filtration $F_\bullet$ is bounded below but not bounded above and is exhaustive
$$C = \bigcup_{p=0}^\infty F_p C.$$
Note also that $C$ is bigraded by $(\ell, i)$
\begin{equation} \label{Cbigradingorig}
C = \bigoplus_{\ell = 0}^{ab} \bigoplus_{i=0}^\infty C^\ell(i).
\end{equation}

It is clear that $d$ may be written as a direct sum $d = d_2 + d_{-2}$ where $d_2$ increases the polynomial degree by two and $d_{-2}$ lowers the polynomial degree by two, see Equation \eqref{defnofd}. From $d^2 = 0$ we obtain
\begin{lem} \label{doublecomplex}
\hfill

\begin{enumerate}
\item $d_2^2 = 0$
\item $d_{-2}^2 =0$
\item $d_2 d_{-2} + d_{-2} d_2 =0$.
\end{enumerate}
\end{lem} 
In particular, we have
\begin{equation} \label{dandfiltrationlevel}
d F_p C^\ell \subset F_{p+2} C^{\ell+1}.
\end{equation}

\begin{rmk}
Relative to the bigrading of $C$ given by Equation \eqref{Cbigradingorig}, $d$ is a bigraded map with
\begin{equation}
d = d_{1,2} + d_{1,-2}.
\end{equation}
\end{rmk}

Since $d$ increases filtration degree, $(F_\bullet, C, d)$ is not a filtered cochain complex.  Also, the above filtration of $C$ is increasing whereas the general theory assumes it is decreasing.  We can, however, correct this by regrading $C$ so that $d$ preserves the filtration and so that the filtration is decreasing.

The vector space underlying the complex $(C, d)$ is bigraded by cochain degree $\ell$ and polynomial degree $p$. As is customary in the theory of spectral sequences, we regrade by complementary degree $q = \ell - p$ and $p$. We change this bigrading according to $(p,q) \rightarrow (p^\prime, q^\prime)$ where $p^\prime = p - 2 \ell$ and $q^\prime = p+q - p^\prime$.  Note that $\ell = p+q = p^\prime + q^\prime$ and $d$ preserves the filtration.  That is,
\begin{equation*}
d F_{p^\prime} C^\ell \subset F_{p^\prime} C^{\ell+1}.
\end{equation*}

The theory of the spectral sequence associated to a filtered cochain complex requires the filtration to be decreasing, however the filtration $F_{p^\prime}$ is increasing.  Accordingly, we pass to the new decreasing filtration $F^{p^{\prime \prime}}$ defined by $F^{p^{\prime \prime}} = F_{-p^{\prime \prime}}$.  As before, $q^{\prime \prime}$ is the complementary degree, so $q^{\prime \prime} = \ell - p^{\prime \prime}$.  Hence
\begin{equation*}
p^{\prime \prime} = 2\ell - p \text{ and } q^{\prime \prime} = p - \ell.
\end{equation*}

Thus, from the bigrading of Equation \eqref{Cbigradingorig} we obtain a new bigrading for $C$
\begin{equation} \label{Cbigrading}
C = \bigoplus_{p^{\prime \prime} = -\infty}^{2ab} \bigoplus_{q^{\prime \prime} = -ab}^{\infty} C^{p^{\prime \prime}, q^{\prime \prime}}.
\end{equation}

\begin{lem} \label{dnewbigrading}
Relative to the bigrading Equation \eqref{Cbigrading}, we have $d = d_{0,1} + d_{4,-3}$.  In particular, $d$ preserves the new filtration.
\end{lem}

\begin{rmk} \label{d=d_{0,1}}
The differential $d^\prime$ on $E_0$ is induced by the summand $d_{0,1}$ in Lemma \ref{dnewbigrading}.  In fact, we use the negative of $d_{0,1}$.  That is, we take
$$d^\prime = \sum_{i=1}^k \sum_{\alpha=1}^a \sum_{\mu=a+1}^{a+b} A(\omega_{\alpha \mu}) \otimes z_{\alpha i} z_{\mu i}.$$
\end{rmk}

$F^{\bullet}$ is a decreasing filtration preserved by $d$ and the associated bigraded vector space $E_0^{p'', q''} = \bigoplus_{p'',q''}F^{p''} C^{p'' + q''}$ is  supported in the quadrant $p'' \leq 2ab$ and $q'' \geq -ab$.  In addition, because the cohomological degree $\ell$ satisfies $ 0 \leq \ell \leq ab$ and $\ell = p'' + q''$, $E_0^{p'', q''}$ is supported in the intersection of the above quadrant with the band $0 \leq p'' + q'' \leq ab$.

In what follows we will abuse notation and write $p$ and $q$ instead of $p''$ and $q''$.  The figure shows the support of the bigraded complex $C^{\bullet,\bullet}$ with the new bigrading.

\begin{figure} \label{supportfigure}
\centering
\begin{tikzpicture}
\draw[draw=gray!50!white,fill=gray!20!white] 
    plot[smooth,samples=100,domain=-2.5:1.5] (\x,{-\x}) -- 
    plot[smooth,samples=100,domain=3:-1] (\x,{1.5-\x});

\draw[draw=gray!50!white,fill=gray!70!white] 
    plot[smooth,samples=100,domain=-1.5:2.5] (1.2,\x) -- 
    plot[smooth,samples=100,domain=2.5:-1.5] (3,\x);

\draw[draw=gray!50!white,fill=gray!110!white] 
    plot[smooth,samples=100,domain=1.2:1.5] (\x,{-\x}) -- 
    plot[smooth,samples=100,domain=3:1.2] (\x,{1.5 - \x});

\node at (1.7,-1) {$F^p$};

\draw[domain=-2.5:2] plot (\x,{-\x});
\node at (-2.1,1.6) {\small{$\ell = 0$}};

\draw[domain=-1:3] plot (\x,{1.5-\x});
\node at (-.41,2.4) {\small{$\ell = ab$}};

\draw[thick, dashed] (-2.5,-1.5) -- (3.5,-1.5);
\node at (-2,-1.7) {$q = -ab$};

\draw[thick, dashed] (3,-2.5) -- (3,2.5);
\node at (3.7,2) {$p = 2ab$};

\draw (-2.5,0)--(3.5,0) node[right]{$p$};
\draw (0,-2.5)--(0,2.5) node[above]{$q$};

\draw (1.2,-2.5)--(1.2,2.5);
\node at (1.8,2) {$F^p \rightarrow$};
\end{tikzpicture}
$$C^{p,q} \text{ is supported in the shaded diagonal region}.$$
\end{figure}

Because the filtration is bounded above by $p = 2ab$, for each $(p,q)$ there exists an $r(p)$ so that $Z_r^{p,q} = Z^{p,q}$ for all $r \geq r(p)$. In our case it suffices to take $r(p) = 2ab - p +1$.  Note that $r(p)$ is a decreasing function of $p$, so $r > r(p)$ implies $r > r(p+k)$ for $k \geq 1$.  Since the complex is bounded below by $q \geq -ab$ we also obtain an analogous bound $r(q) = q+ab+1$, however, we will use only $r(p)$.
\begin{rmk}
For the action of a general reductive $G$ on the polynomial Fock model, the exterior differential may be decomposed as $d = d_{-2} + d_0 + d_2$ and preserves the new filtration.  The support of the resulting bigraded complex will still be contained in the band of the figure on the previous page. We leave the details to the reader.  
\end{rmk}

Now that we have finished any details concerning the filtration used to obtain a spectral sequence, we return to using $p,q$ for the signature of a real vector space.  We now prove a general theorem about the relative Lie algebra cohomology of $\SO(p,q)$ with values in $\mathcal{P}_k$ for large $k$.





\section{The cohomology of $\SO(p,q)$ when $k \geq pq$.} \label{kgeqnsection}
In this section, we work under the assumption
\begin{equation}
k \geq pq.
\end{equation}
Let $C^\ell = C^\ell(\mathfrak{so}(p,q), \SO(p) \times \SO(q); \mathcal{P}_k)$.  We prove the following theorem, see Equation \eqref{qalphadef} for the definition of the quadratic elements $q_{\alpha \mu} \in \mathcal{P}(V^k)$.
\begin{thm} \label{kgeqnvanishing}
Assume $k \geq pq$.  Then we have 
\begin{equation*}
H^\ell(\mathfrak{so}(p,q), \SO(p) \times \SO(q); \mathcal{P}_k) = \begin{cases}
0 &\text{ if } \ell \neq pq \\
(\mathcal{P}_k/(q_{1, p+1}, \ldots, q_{p,p+q}))^K \vol &\text{ if } \ell = pq.
\end{cases}
\end{equation*}
\end{thm}

Our method of proof is to compute the cohomology of the $E_0$ term of the spectral sequence we have just developed.  To do this, we will first compute the cohomology of this complex before taking $K$-invariance.  This complex will be a Koszul complex associated to a regular sequence.  We will then use the results of Section \ref{spectralsection} and Equation \eqref{E_0CisA} to finish the calculation.

Define the complex $(A, d_A)$ by
\begin{equation} \label{defofAcomplex}
A^\ell = \wwedge{\ell}(\mathfrak{p}^*) \otimes \mathcal{P}_k \text{ and } d_A = \sum_{\alpha = 1}^p \sum_{\mu = p+1}^{p+q} \sum_{i=1}^k  A(\omega_{\alpha \mu}) \otimes z_{\alpha i} z_{\mu i}.
\end{equation}
Then $C^\ell = (A^\ell)^K$ and by Remark \ref{d=d_{0,1}} we have, since $K$ is compact,
\begin{equation} \label{E_0CisA}
H^\ell(E_0(C)) = (H^\ell(A))^K.
\end{equation}

\section{Koszul complexes and regular sequences} \label{defofkoszulsection}
We will see that $E_0(C)$ is a Koszul complex.  We now define what a Koszul complex is and state some results about Koszul complexes in case the defining elements form a regular sequence.

Let $S = \C[x_1, \ldots, x_n]$ and $f_1, \ldots, f_N \in S$.  Let $Y = S^N$, $e_1, \ldots, e_N$ be the standard basis for $Y$, and $\omega_1, \ldots, \omega_N$ be the dual basis for $\mathrm{Hom}_S(Y, S)$.  Define the Koszul complex $K(f_1, \ldots, f_N)$ by
\begin{align*}
K^\ell &= {\bigwedge}^\ell Y^* \\
d &= \sum_i f_i A(\omega_i).
\end{align*}

From Eisenbud, \cite{Eisenbud} Corollary 17.5, we have
\begin{prop}
If $f_1, \ldots, f_N$ is a regular sequence in $S$ then
\begin{enumerate}
\item $H^\ell(K(f_1, \ldots, f_N)) = 0$ for $\ell < N$
\item $H^N(K(f_1, \ldots, f_N)) \cong S/(f_1, \ldots, f_N)$.
\end{enumerate}
\end{prop}

In fact, we also have the following result of \cite{Eisenbud}, Corollary 17.12.

\begin{lem} \label{Eisenbuddepth}
If $x_1, \ldots, x_i$ is a regular sequence then
\begin{equation*}
H^\ell(K(x_1, \ldots, x_N)) = 0 \text{ for } \ell < i.
\end{equation*}

\end{lem}

We now show that $(A, d_A)$ is a Koszul complex, and hence $E_0(C)$ is a sub-Koszul complex.

\subsection{$E_0(C)$ is a Koszul complex.} \label{E_0isakoszulsection}

Define the quadratic elements $q_{\alpha \mu}$ of $\mathcal{P}_k$ by
\begin{equation} \label{qalphadef}
q_{\alpha \mu} = \sum_{i=1}^k z_{\alpha i} z_{\mu i} \text{ for } 1 \leq \alpha \leq p, p+1 \leq \mu \leq p+q.
\end{equation}
We note that the $q_{\alpha \mu}$ are the result of the following matrix multiplication of elements of $\mathcal{P}_k$
\begin{align*}
\begin{pmatrix}
z_{11} & z_{12} & \cdots & z_{1k} \\
\vdots & \vdots & \ddots & \vdots \\
z_{p1} & z_{p2} & \cdots & z_{pk}
\end{pmatrix}&
\begin{pmatrix}
z_{p+1,1} & z_{p+2,1} & \cdots & z_{p+q,1} \\
\vdots & \vdots & \ddots & \vdots \\
z_{p+1 k} & z_{p+2,k} & \cdots & z_{p+q,k}
\end{pmatrix} \\
=&
\begin{pmatrix}
q_{1,p+1} & q_{1,p+2} & \cdots & q_{1,p+q} \\
\vdots & \vdots & \ddots & \vdots \\
q_{p,p+1} & q_{p,p+2} & \cdots & q_{p,p+q}
\end{pmatrix}.
\end{align*}
Then we have
\begin{equation*}
d_A = \sum_{\alpha = 1}^p \sum_{\mu = p+1}^{p+q} A(\omega_{\alpha \mu}) \otimes q_{\alpha \mu}.
\end{equation*}

We first note that $d_A$ is the differential in the Koszul complex $K(\{q_{\alpha, \mu}\})$ associated to the sequence of the quadratic polynomials $q_{\alpha \mu}$, see Eisenbud \cite{Eisenbud}, Section 17.2.  To see that the Koszul complex as described in \cite{Eisenbud} is the above complex $A$ we choose $\mathcal{P}_k$ as Eisenbud's ring $R$ and $\mathcal{P}_k^k$ as Eisenbud's module $N$. In our description, since $\mathfrak{p} \cong V_+ \otimes V_- \cong \C^{pq}$, we are using the exterior algebra $\bigwedge^{\bullet}( (\C^{pq})^*) \otimes \mathcal{P}_k$.  But the operation of taking the exterior algebra of a module commutes with base change and hence we have $\bigwedge^{\bullet}( (\C^{pq})^*) \otimes \mathcal{P}_k \cong \bigwedge^{\bullet} (\mathcal{P}_k^{pq})$. Then we apply Eisenbud's construction with the sequence $\{q_{\alpha \mu}\}$ to obtain the above complex $A$.  We recall that $f_1, \ldots, f_n$ is a regular sequence in a ring $R$ if and only if $f_i$ is not a zero divisor in $R/(f_1, \ldots, f_{i-1})$ for $1 \leq i \leq n$.  We will show the $q_{\alpha \mu}$ form a regular sequence.  This will be a consequence of the following two lemmas.

The following lemma is Matsumura's corollary to Theorem 16.3 on page 127, \cite{Matsumura}.  It gives a condition under which we may reorder a sequence while preserving regularity.

\begin{lem} \label{Matsumuralemma}
If $R$ is Noetherian and graded and $a_1, \ldots, a_n$ is a regular sequence of homogeneous elements in $R$, then so is any permutation of $a_1, \ldots, a_n$.
\end{lem}

\begin{rmk}
The above lemma allows us to say, for instance, that the $\{q_{\alpha \mu}\}$ are a regular sequence without having to order the elements.
\end{rmk}

\begin{lem} \label{easyregular}
Let $R = \C[x_1, \ldots, x_N, y_1, \ldots, y_N]$.  Then $(x_1 y_1, x_2 y_2, \ldots, x_N y_N)$ is a regular sequence.
\end{lem}

\begin{proof}
We first rewrite $R$ as the tensor product of $N$ polynomial rings
\begin{equation*}
R \cong \C[x_1, y_1] \otimes \C[x_2, y_2] \otimes \cdots \otimes \C[x_N, y_N].
\end{equation*}
Fix $i$ between $1$ and $N$.  We verify that $x_i y_i$ is not a zero divisor in
$$R_i = R/(x_1 y_1, \ldots, x_{i-1} y_{i-1}).$$
Note that
\begin{equation*}
R_i \cong \frac{\C[x_1, y_1]}{(x_1 y_1)} \otimes \frac{\C[x_2, y_2]}{(x_2 y_2)} \otimes \frac{\C[x_{i -1}, y_{i -1}]}{(x_{i -1} y_{i -1})} \otimes \C[x_{i}, y_{i}] \otimes \cdots \otimes \C[x_N, y_N].
\end{equation*}
Let $b_{e,f} = x_i^e y_i^f \in \C[x_i, y_i]$.  Then $\{b_{e,f}\}_{e,f \geq 0}$ is a basis for $\C[x_i, y_i]$.  Now consider the map
\begin{align*}
g : R_i &\to R_i \\
r &\mapsto x_i y_i r.
\end{align*}
We show $g$ is injective and thus $x_i y_i$ is not a zero divisor in $R_i$.  Suppose $r$ is in the kernel of $g$.  Then $r$ has a unique representation
\begin{equation*}
r = \sum_{e,f} a_{1,e,f} \otimes a_{2,e,f} \otimes \cdots \otimes b_{e,f} \otimes a_{i+1,e,f} \otimes \cdots \otimes a_{N,e,f} \text{ where } a_{\beta,e,f} \in \C[x_\beta, y_\beta].
\end{equation*}
Hence
\begin{align*}
g(r) &= \sum_{e,f} a_{1,e,f} \otimes a_{2,e,f} \otimes \cdots \otimes x_i y_i b_{e,f} \otimes \cdots \otimes a_{N,e,f} \\
&= \sum_{e,f} a_{1,e,f} \otimes a_{2,e,f} \otimes \cdots \otimes b_{e+1,f+1} \otimes \cdots \otimes a_{N,e,f} = 0.
\end{align*}
Since $b_{e+1, f+1}$ is a basis for $\C[x_i, y_i]$, we have, for all $e,f \geq 0$,
\begin{equation*}
a_{1,e,f} \otimes a_{2,e,f} \otimes \cdots \otimes a_{i-1,e,f} \otimes a_{i+1,e,f} \otimes \cdots \otimes a_{N,e,f} = 0.
\end{equation*}
Hence $r = 0$.  Thus $x_i y_i$ is not a zero divisor and the lemma is proved.

\end{proof}

\begin{rmk}
The above proof can be adapted to show that any ``disjoint" monomials form a regular sequence.  The details are left to the reader.
\end{rmk}

\begin{prop} \label{Aregularsequence}
The $q_{\alpha \mu}$ form a regular sequence in $\mathcal{P}_k$.
\end{prop}

\begin{proof}
First, we examine a longer sequence $\sigma$ where we have prepended many of the variables $z_{\alpha i}$ and $z_{\mu i}$ to the sequence of $q_{\alpha \mu}$.  We will show this is a regular sequence.  Once we have done this, by Lemma \ref{Matsumuralemma} we can reorder so that the $q_{\alpha \mu}$ come first and this reordered sequence will still be regular.  Then we use the obvious fact that any initial segment of a regular sequence is a regular sequence and hence the $q_{\alpha \mu}$ form a regular sequence.

We first define the sequence $\tau$ as follows.  It will contain all the variables $z_{\alpha i}$ except for those with $i \equiv \alpha \mathrm{ mod } p$.  It will also contain all the $z_{\mu i}$ except for those with $(\mu - p-1)p < i \leq (\mu - p)p$.  Now we define $\sigma$ to be $\tau$ followed by the $q_{\alpha \mu}$.  We now check that this is a regular sequence.

It is clear that $\tau$ is a regular sequence.  The ``off-diagonal" (see the diagram below) $z_{\alpha i}$ and the $z_{\mu i}$ included form a regular sequence since they are coordinates.  To check that $\sigma$ is regular, we must check that the $q_{\alpha \mu}$ are a regular sequence in $\mathcal{P}_k / (\tau)$.  Note that
\begin{equation*}
\mathcal{P}_k / (\tau) \cong \C[\{z_{\alpha i}\}_{i \equiv \alpha \mathrm{ mod }p} \cup \{z_{\mu i}\}_{p(\mu-p-1) < i \leq p(\mu - p) } ].
\end{equation*}
The image of $q_{\alpha \mu}$ in this quotient ring is the result of the matrix multiplication
\begin{align*}
&\begin{pmatrix}
z_{11} & 0 & \cdots & 0 & \cdots & \cdots & z_{1, p(q-1) + 1} & 0 & \cdots & 0\\
0 & z_{22} & \cdots & 0 & \cdots & \cdots & 0 & z_{2, p(q-1) + 2} & \cdots & 0\\
\vdots & \vdots & \boldsymbol{\ddots} & \vdots & \cdots &\cdots & \vdots & \vdots & \boldsymbol{\ddots} & \vdots\\
0 & 0 & \cdots & z_{pp} & \cdots & \cdots & 0 & 0 & \cdots & z_{p, pq}\\
\end{pmatrix} \bullet \\
&\begin{pmatrix}
z_{p+1,1} & 0 & \cdots & 0 \\
\boldsymbol{\vdots} & \vdots & \ddots & \vdots \\
z_{p+1,p} & 0 & \cdots & 0 \\
0 & z_{p+2,p+1} & \cdots & 0 \\
\vdots & \boldsymbol{\vdots} & \ddots & \vdots \\
0 & z_{p+2,2p} & \cdots & 0 \\
\vdots & \vdots & \boldsymbol{\ddots} & \vdots \\
0 & 0 & \cdots & z_{p+q, p(q-1)+1} \\
\vdots & \vdots & \ddots & \boldsymbol{\vdots} \\
0 & 0 & \cdots & z_{p+q, pq}
\end{pmatrix}.
\end{align*}
That is,
\begin{equation*}
q_{\alpha \mu} \mapsto z_{\alpha, (\mu-p-1)p + \alpha} z_{\mu, (\mu-p-1)p + \alpha}.
\end{equation*}
These elements form a regular sequence by Lemma \ref{easyregular}.

Now we apply Lemma \ref{Matsumuralemma} to reorder $\sigma$ such that the $q_{\alpha \mu}$ come first, and note that any initial segment of $\sigma$, in particular the $\{q_{\alpha \mu}\}$, is a regular sequence.

\end{proof}

\begin{rmk}
Proposition \ref{Aregularsequence} provides an upper bound on the minimal $k$ so that the $q_{\alpha \mu}$ form a regular sequence.  A lower bound is $p$ since if $k < p$ then the form $\varphi_{kq}$ of Kudla and Millson is a non-zero cohomology class in degree $kq$ (which is less than the top degree $pq$).
\end{rmk}

We now compute the cohomology of $(A, d_A)$
\begin{prop} \label{Avanishing}
\begin{equation*}
H^\ell(A) = \begin{cases}
0 &\text{ if } \ell \neq pq \\
\mathcal{P}_k/(\{q_{\alpha \mu}\}) \vol &\text{ if } \ell = pq.
\end{cases}
\end{equation*}
\end{prop}

\begin{proof}
Corollary 17.5 of \cite{Eisenbud} (with $M = R$), states that the cohomology of a Koszul complex $K(f_1,\ldots,f_N)$ below the top degree vanishes if $f_1,\ldots,f_N$ is a regular sequence and that in this case the top cohomology $H^N(K(f_1, \ldots, f_N))$ is isomorphic to $R/(f_1, \ldots, f_N)$.

\end{proof}

We now prove Theorem \ref{kgeqnvanishing}.

\begin{proof}

By Equation \eqref{E_0CisA} and Proposition \ref{Avanishing}, we have, since $\vol$ is $K$-invariant,
\begin{equation*}
H^\ell(E_0(C)) = \begin{cases}
0 &\text{ if } \ell \neq pq \\
(\mathcal{P}_k/(\{q_{\alpha \mu}\}))^K \vol &\text{ if } \ell = pq.
\end{cases}
\end{equation*}
Thus, by Proposition \ref{grCzeroimpliesCzero} and statement $(3)$ of Proposition \ref{generalspectral}, we have
\begin{equation*}
H^\ell(\mathfrak{so}(p,q), \SO(p) \times \SO(q); \mathcal{P}_k) = \begin{cases}
0 &\text{ if } \ell \neq pq \\
(\mathcal{P}_k/(\{q_{\alpha \mu}\}))^K \vol &\text{ if } \ell = pq.
\end{cases}
\end{equation*}

\end{proof}

We now show that $H^{pq}(C)$ is not finitely generated.

\begin{prop}
The map from $\mathrm{Pol}(V_+^{pq})^{\SO(p)} \oplus \mathrm{Pol}(V_-^{pq})^{\SO(q)}$ to $H^{pq}(C)$ sending $(f,g)$ to $[(f + g)\vol]$ is an injection.
\end{prop}

\begin{proof}
First, note that $f+g$ is $\SO(p) \times \SO(q)$-invariant, hence $(f+g) \vol$ is $K$-invariant.  Now we show the map is an injection.  Suppose $f+g \in ( \{q_{\alpha \mu} \})$.  Then since $( \{q_{\alpha \mu} \}) \subset (\{z_{\alpha i} \})$ we have $g = 0$.  Similarly, since $( \{q_{\alpha \mu} \}) \subset (\{z_{\mu i} \})$ we have $f = 0$.  Hence the map is injective.


\end{proof}

\section{The cohomology of $\SO(p,q)$ when $p \geq kq$.}

In this section we prove the following theorem

\begin{thm} \label{mainpgeqkq}
If $p \geq kq$, then
\begin{equation*}
H^\ell(\mathfrak{so}(p,q), \SO(p) \times \SO(q); \mathcal{P}_k) = \begin{cases}
\text{nonzero } &\text{ if } \ell = kq, pq \\
\text{unknown } &\text{ if } kq < \ell < pq \\
0 &\text{ otherwise}.
\end{cases}
\end{equation*}

\end{thm}

We use the same proof technique as in the previous section together with Lemma \ref{Eisenbuddepth}.

In what follows we use the symbol $C$ to denote the complex \\
$C^\bullet(\mathfrak{so}(p,q), \SO(p) \times \SO(q); \mathcal{P}_k)$.  We now prove the following proposition by finding a regular subsequence of the $q_{\alpha \mu}$ of length $kq$.  Let $(A, d_A)$ be the complex defined in Equation \eqref{defofAcomplex}.

\begin{prop}
If $p \geq kq$, then
\begin{equation*}
H^\ell(A) = \begin{cases}
0 &\text{ if } \ell < kq\\
\text{nonzero } &\text{ if } \ell = kq.
\end{cases}
\end{equation*}
\end{prop}

\begin{proof}
We first construct a regular subsequence of the $q_{\alpha \mu}$ of length $kq$.  Consider the subsequence $\tau = \{q_{\alpha, p + \left \lfloor {\frac{\alpha}{k}}\right \rfloor} \}_{1 \leq \alpha \leq kq}$ where $\left \lfloor {\frac{\alpha}{k}}\right \rfloor$ is the greatest integer less than or equal to $\frac{\alpha}{k}$.  We claim this is a regular sequence in $\mathcal{P}_k$.  Similar to the proofs before, we prepend the ``off-diagonal'' variables $z_{\alpha, i}$ with $\alpha \not \equiv i \text{ mod }k$.  That is, we consider the sequence
$$\sigma = \{{z_{\alpha, i}}_{\alpha \not \equiv i \text{ mod }k}, \tau\}.$$
Clearly the initial segment is regular.  To check that $\sigma$ is regular it remains to show $\tau$ is a regular sequence in $\mathcal{P}_k/ (\{{z_{\alpha, i}}_{\alpha \not \equiv i \text{ mod }k}\})$.  The image of $\tau$ in this quotient ring is the result of the following matrix multiplication,

\begin{equation*}
\begin{pmatrix}
z_{1,1} & 0 & \cdots & 0 \\
0 & z_{2,2} & \cdots & 0 \\
\vdots & \vdots & \boldsymbol{\ddots} & \vdots \\
0 & 0 & \cdots & z_{k,k} \\
z_{k+1,1} & 0 & \cdots & 0 \\
0 & z_{k+2,2} & \cdots & 0 \\
\vdots & \vdots & \boldsymbol{\ddots} & \vdots \\
0 & 0 & \cdots & z_{2k,k} \\
\vdots & \vdots & \vdots & \vdots \\
\end{pmatrix} \bullet 
\begin{pmatrix}
z_{p+1,1} & z_{p+2,1} & \cdots & z_{p+q,1} \\
\vdots & \vdots & \cdots & \vdots \\
z_{p+1,k} & z_{p+2,k} & \cdots & z_{p+q,k}
\end{pmatrix}.
\end{equation*}
That is, the image of $\tau$ is $\{ z_{\alpha, i} z_{p+1+ \left \lfloor \frac{\alpha}{k} \right \rfloor,i} \}_{\alpha \equiv i \text{ mod }k}$.  Thus the image of $\tau$ is a sequence of disjoint quadratic monomials, hence regular by Lemma \ref{easyregular}.  As we did in our previous proof of regularity, we reorder this sequence using Lemma \ref{Matsumuralemma} to get a regular sequence with these $q_{\alpha \mu}$ first and note than any initial segment of a regular sequence is regular.  Thus, by Lemma \ref{Eisenbuddepth}, we have
\begin{equation*}
H^\ell(A) = 0 \text{ for } \ell < kq.
\end{equation*}
To show that $H^{kq}(A) \neq 0$, we note that the form $\varphi_{kq}$, defined in Equation \eqref{defnofKMkq}, of Kudla and Millson is closed and not exact since it only involves positive variables (and any exact form necessarily has negative coordinates).

\end{proof}

Now we prove Theorem \ref{mainpgeqkq}.

\begin{proof}
Since $H^\ell(A) = 0$ for $\ell < kq$, by Proposition \ref{grCzeroimpliesCzero} and Equation \eqref{E_0CisA}, we have
$$H^\ell(C) = 0 \text{ for } \ell < kq.$$
Since $\varphi_{kq}$ is closed in $C$, $\SO(p) \times \SO(q)$ invariant, and the $(kq-1)^{st}$ cohomology group of the associated graded of $C$ vanishes, we find that $\varphi_{kq}$ is not exact in $E_\infty$ and so not exact in C.

\end{proof}

%% file: Chapter4.tex

\renewcommand{\thechapter}{4}

\chapter{The Relative Lie Algebra Cohomology for $\SO_0(n,1)$}

In this chapter, because the cohomology of the complex ``before taking invariance'', $A$, is large, we first compute the invariant cochains and then compute the cohomology.  The $K$-invariant cochains of $A$ form a subcomplex, $(A)^K$, which is again a Koszul complex.  The main point of this chapter is to compute the cohomology of this subcomplex.

\section{Introduction}
\bigskip
We let $(V, (,))$ be Minkowski space $\R^{n,1}$ and $e_1, e_2, \ldots, e_{n+1}$ be the standard basis.  Let $V_+$ be the span of $e_1, \ldots, e_n$.  We will consider the connected real Lie group $G = \SO_0(n,1)$ with Lie algebra $\mathfrak{so}(n,1)$ and maximal compact subgroup $K=\SO(n)$ with Lie algebra $\mathfrak{so}(n)$, the subgroup of $G$ that fixes the last basis vector $e_{n+1}$.  Let $\mathcal{S}_k$ be the $\OO(n)$-invariant complex-valued polynomials on $V^k$ and $\mathcal{R}_k \subset \mathcal{S}_k$ be the $\OO(n)$-invariant complex-valued polynomials on $V_+^k$, see Section \ref{notation}.  We will consider the Weil representation with values in the Fock model $\mathcal{P}_k = \mathrm{Pol}((V \otimes \C)^k)$.

We restate Theorem \ref{main}, our results for the cohomology with values in $\mathcal{P}_k$.  In the following theorem, let $\varphi_k$ be the cocycle constructed in the work of Kudla and Millson, \cite{KM2}, see Section \ref{notation}, Equation \eqref{vaprhikdef}.  In what follows, $c_1, \ldots, c_k$ are the cubic polynomials on $V^k$ defined in Equation \eqref{cubic}, $q_1, \ldots, q_n$ are the quadratic polynomials on $V^k$ defined in Equation \eqref{qalphadef}, and $\vol$ is as defined in Equation \eqref{voldefn}.  Note that statement $(3)$ is a consequence of Theorem \ref{kgeqnvanishing} and hence we will eventually assume that $k < n$ since these are the only cases remaining.

\begin{thm}
\hfill

\begin{enumerate}
\item  If $k < n$ then
\begin{equation*}
H^\ell \big(\mathfrak{so}(n,1) ,\SO(n); \mathcal{P}_k \big) = \begin{cases}
\mathcal{R}_k \varphi_k &\text{ if } \ell=k \\
\mathcal{S}_k/(c_1, \ldots, c_k) \vol &\text{ if } \ell=n \\
0 &\text{ otherwise }
\end{cases}
\end{equation*}
\item If $k = n$ then
\begin{equation*}
H^\ell \big(\mathfrak{so}(n,1) ,\SO(n); \mathcal{P}_k \big) = \begin{cases}
\mathcal{R}_k \varphi_k \oplus \mathcal{S}_k/(c_1, \ldots, c_k) \vol &\text{ if } \ell=n \\
0 &\text{ otherwise }
\end{cases}
\end{equation*}
\item If $k>n$
\begin{equation*}
H^\ell \big(\mathfrak{so}(n,1) ,\SO(n); \mathcal{P}_k \big) = \begin{cases}
\big( \mathcal{P}_k/(q_1, \ldots, q_n) \big)^K \vol &\text{ if } \ell=n \\
0 &\text{ otherwise }
\end{cases}
\end{equation*}
\end{enumerate}
\end{thm}

The cohomology groups $H^\ell \big(\mathfrak{so}(n,1) ,\SO(n); \mathcal{P}_k \big)$ are $\mathfrak{sp}(2k, \R)$-modules.  We now describe these modules.  If $k < \frac{n+1}{2}$, then as an $\mathfrak{sp}(2k, \R)$-module $\mathcal{R}_k \varphi_k$ is isomorphic to the space of $\mathrm{MU}(k)$-finite vectors in the holomorphic discrete series representation with parameter $(\frac{n+1}{2}, \cdots,\frac{n+1}{2}) $. If $k < n$, then the cohomology group $H^k\big(\mathfrak{so}(n,1) ,\SO(n); \mathcal{P}_k \big)$ is an irreducible holomorphic representation because it was proved in \cite{KM2} that the class of $\varphi_k$ is a lowest weight vector in $H^k\big(\mathfrak{so}(n,1) ,\SO(n); \mathcal{P}_k \big)$.  On the other hand, the cohomology group \\
$H^n\big(\mathfrak{so}(n,1) ,\SO(n); \mathcal{P}_k \big)$ is never irreducible.  Indeed, if $k < n$ then \\
$H^n\big(\mathfrak{so}(n,1) ,\SO(n); \mathcal{P}_k \big)$ is the direct sum of two nonzero $\mathfrak{sp}(2k,\R)$-modules $H_+^n$ and $H_-^n$ and if $k = n$, then $H^n\big(\mathfrak{so}(n,1) ,\SO(n); \mathcal{P}_k \big)$ is the direct sum of three nonzero $\mathfrak{sp}(2k,\R)$-modules $H_+^n$, $H_-^n$, and $\mathcal{R}_k(V) \varphi_k$.



\section{The relative Lie algebra complex} \label{notation}

We reestablish notation that we use throughout this chapter.  Let $e_1,\ldots, e_{n+1}$ be an orthogonal basis for $V$ such that $(e_i,e_i) =1, 1 \leq i \leq n$ and $(e_{n+1},e_{n+1}) = -1$. We let $x_1,\ldots,x_n, t$ be the corresponding coordinates. 

We have the splitting (orthogonal for the Killing form)
$$\mathfrak{so}(n,1) = \mathfrak{so}(n) \oplus \mathfrak{p}_0$$
and denote the complexification $\mathfrak{p}_0 \otimes \C$ by $\mathfrak{p}$.

We recall that the map $\phi:\bigwedge^2(V) \to \mathfrak{so}(n,1)$ given by
$$\phi(u \wedge v)(w) = (u,w)v - (v,w)u $$
is an isomorphism.  Under this isomorphism the elements $e_i \wedge e_{n+1}, 1 \leq i \leq n$ are a basis for $\mathfrak{p}_0$.  We define $e_{i,n+1} = -e_i \wedge e_{n+1}$ and let $\omega_1,\ldots,\omega_n$ be the dual basis for $\mathfrak{p}_0^*$.  We let $\mathcal{I}_{\ell,n}$ ($\mathcal{I}$ for injections) be the set of all ordered $\ell$-tuples of distinct elements $I= (i_1,i_2,\ldots,i_{\ell})$ from $\{1,2,\ldots,n\}$, that is, the set of all injections from the set $\{1, \ldots, \ell\}$ to $\{1, \ldots, n\}$.  We let $\mathcal{S}_{\ell, n} \subset \mathcal{I}_{\ell,n}$ ($\mathcal{S}$ for stricly increasing) be the subset of strictly increasing $\ell$-tuples.  We define $\omega_I $ for $I \in \mathcal{I}_{\ell,n}$ by
$$\omega_I = \omega_{i_1} \wedge \cdots \wedge \omega_{i_{\ell}}.$$
Then the element $\vol \in (\wwedge{n}\mathfrak{p}_0^*)^K$ is, by Equation \eqref{voldefn} given by
\begin{equation}
\vol = \omega_1 \wedge \cdots \wedge \omega_n.
\end{equation}

Now let $V^k = \displaystyle \bigoplus_{i=1}^k V$.  We will use $\mathbf{v}$ to denote the element $(v_1,v_2,\cdots,v_k) \in V^k$.  We will often identify $V^k$ with the $((n+1) \times k)$-matrices 
$$\begin{pmatrix}
x_{11} & x_{12} & \cdots & x_{1k} \\
\vdots & \vdots & \ddots & \vdots \\
x_{n1} & x_{n2} & \cdots & x_{nk} \\
t_1 & t_2 & \cdots & t_k
\end{pmatrix}$$
over $\R$ using the basis $e_1, \ldots, e_{n+1}$.  Then $\mathbf{v}$ will correspond to the $(n+1) \times k$ matrix $X$ where $v_j$ is the $j^{th}$ column of the matrix.

Let $V_+$ be the span of $e_1, \ldots, e_n$ and $V_-$ be the span of $e_{n+1}$.  Then we have the splitting $V = V_+ \oplus V_-$ and the induced splitting $V^k = V_+^k \oplus V_-^k$.  We define, for $1 \leq i,j \leq k$, the quadratic function $r_{ij} \in \mathrm{Pol}(V_+^k)$ for $\mathbf{v} \in V_+^k$ by
\begin{equation}
r_{ij}(\mathbf{v}) = (v_i,v_j).
\end{equation}


We let $\mathcal{R}_k = \mathcal{R}_k(V_+)$ be the subalgebra of $\mathrm{Pol}(V_+^k)$ generated by the $r_{ij}$ for $1 \leq i,j \leq k$.  We note that 
$$\mathcal{R}_k = \mathrm{Pol}(V_+^k)^{\mathrm{O}(n)}$$
is the algebra of polynomial invariants of the group $\mathrm{O}(n)$ (the ``First Main Theorem'' for the orthogonal group, \cite{Weyl}, page 53).  Since (as a consequence of our assumption \eqref{k<n/2} immediately below) we will have $k < n$, it is a polynomial algebra.  This follows from the ``Second Main Theorem'' for the orthogonal group, \cite{Weyl}, page 75. We will assume that
\begin{equation} \label{k<n/2}
k < n
\end{equation}
for the remainder of this chapter.  

For $K = \SO(n), \OO(n)$ (embedded in $\SO(n,1)$ fixing the last coordinate), or $\OO(n) \times \OO(1)$, any complex $\mathfrak{so}(n, 1) \times K$-module $\mathcal{V}$, and $\rho : \OO(n,1) \to \mathrm{End}(\mathcal{V})$, we define
$$C^{\bullet}\big(\mathfrak{so}(n,1) ,K; \mathcal{V}\big)$$
by $C^i\big(\mathfrak{so}(n,1) ,K; \mathcal{V}\big) = (\wwedge{i}\mathfrak{p}_0^* \otimes \mathcal{V})^{K}$ and $d = \sum A(\omega_\alpha) \otimes \rho(e_{\alpha} \wedge e_{n+1})$ as in Borel Wallach \cite{BW}.  Throughout this chapter, the symbol $C$ will denote the complex $C^{\bullet} \big(\mathfrak{so}(n,1) ,\SO(n); \mathcal{P}_k \big)$.

\subsection{The relative Lie algebra complex with values in the Fock model}

Similar to the above identification, we identify $(V \otimes \C)^k$ with the $((n+1) \times k)$-matrices
$$\begin{pmatrix}
z_{11} & z_{12} & \cdots & z_{1k} \\
\vdots & \vdots & \ddots & \vdots \\
z_{n1} & z_{n2} & \cdots & z_{nk} \\
w_1 & w_2 & \cdots & w_k
\end{pmatrix}$$
over $\C$ using the basis $e_1, \ldots, e_{n+1}$.  Then $\mathbf{v}$ will correspond to the $(n+1) \times k$ matrix $Z$ where $v_j$ is the $j^{th}$ column of the matrix.

The earlier splitting $V = V_+ \oplus V_-$ induces the splitting $(V \otimes \C)^k = (V_+ \otimes \C)^k \oplus (V_- \otimes \C)^k$.  By abuse of notation, we define, for $1 \leq i,j \leq k$, the quadratic function $r_{ij} \in \mathrm{Pol}((V_+ \otimes \C)^k)$ for $\mathbf{v} \in (V_+ \otimes \C)^k$ by
\begin{equation} \label{r_{ij}}
r_{ij}(\mathbf{v}) = (v_i,v_j).
\end{equation}
Here $(\ ,\ )$ denotes the complex bilinear extension of $(\ ,\ )$ to $V \otimes \C$.

We let $\mathcal{R}_k(V_+ \otimes \C) = \mathrm{Pol}((V_+ \otimes \C)^k)^{\OO(n,\C)}$.  We will abuse notation and sometimes use the symbol $\mathcal{R}_k$ in place of $\mathcal{R}_k(V_+ \otimes \C)$.  The meaning of $\mathcal{R}_k$ should be clear from context.  We note that the $r_{ij}, 1 \leq i,j \leq k$, generate $\mathcal{R}_k$.  

In the following we will be concerned with the relative Lie algebra complex where 
$$C^\ell\big(\mathfrak{so}(n,1) ,\SO(n); \mathcal{P}_k \big) \cong \big( \wwedge{\ell}(\mathfrak{p}_0^*) \otimes \mathcal{P}_k \big)^{\SO(n)}$$
and $d = \sum d^{(j)}$ with

\begin{equation} \label{defofd}
d^{(j)} = \sum_{\alpha=1}^n A(\omega_\alpha) \otimes (\frac{\partial^2}{\partial z_{\alpha,j} \partial w_j} - z_{\alpha,j} w_{j}), \quad 1 \leq j \leq k.
\end{equation}

Recall that 
$$C^{\bullet} \big(\mathfrak{so}(n,1) ,\SO(n); \mathcal{P}_k \big) \cong C^{\bullet} \big(\mathfrak{so}(n,1, \C) ,\SO(n, \C); \mathcal{P}_k \big)$$
where $C^{\bullet} \big(\mathfrak{so}(n,1, \C) ,\SO(n, \C); \mathcal{P}_k \big) = \big( \wwedge{\bullet}\mathfrak{p}^* \otimes \mathcal{P}_k \big)^{\SO(n,\C)}$.

Note that there is the tensor product map $\mathcal{P}_a \otimes \mathcal{P}_b \to \mathcal{P}_{a+b}$ given by 
$$(f_1 \otimes f_2)(\mathbf{v}) = f_1(v_1,\cdots,v_a) f_2(v_{a+1},\cdots,v_{a+b}).$$

The tensor product map induces a bigraded product
\begin{equation} \label{outerwedgedef}
C^i\big(\mathfrak{so}(n,1) ,\SO(n); \mathcal{P}_a \big) \otimes C^j\big(\mathfrak{so}(n,1) , \SO(n); \mathcal{P}_b \big) \to C^{i+j}\big(\mathfrak{so}(n,1) ,\SO(n); \mathcal{P}_{a+b} \big)
\end{equation}
which we will call  the outer exterior  product and denote $\wedge$ given by

\begin{equation*}
(\omega_{I} \otimes f_I) \wedge (\omega_J \otimes f_J) = (\omega_I \wedge \omega_J)  \otimes (f_I \otimes f_J) = (\omega_I \wedge \omega_J) \otimes f_I f_J.
\end{equation*}

We also have, for $f \in \mathcal{P}_k$, the usual multiplication of functions
$$f(\omega_I \otimes f_I) = \omega_I \otimes f f_I.$$

A key point in the computation of the $k$-coboundaries is a product rule for $d$ relative to the outer exterior product. To state this suppose $\psi$ is an outer exterior product 
$$\psi(\bfv) = \psi_1(v_1) \wedge \psi_2(v_2) \wedge \cdots \wedge \psi_k(v_k)$$
where $\mathrm{deg} (\psi_j) =c_j$.  Then we have
\begin{equation} \label{productrule}
 d \psi = \sum_{i=1}^k (-1)^{\sum_{j=1}^{i -1}c_j} \psi_1 \wedge \cdots \wedge d^{(i)} \psi_i \wedge \cdots \wedge \psi_k.
\end{equation}

We define the cocycle $\varphi_1 \in  C^1\big(\mathfrak{so}(n,1), \SO(n); \mathcal{P}_1 \big)$ by 
$$\varphi_1 = \sum_{\alpha=1}^n \omega_\alpha \otimes z_\alpha.$$
We then define $\varphi_k \in C^k(\mathfrak{so}(n,1) ,\SO(n); \mathcal{P}_k )$ by
\begin{equation} \label{vaprhikdef}
\varphi_k = \varphi_1^{(1)} \wedge \cdots  \wedge \varphi_1^{(k)}.
\end{equation}
Here a superscript $(i)$ on a term in a $k$-fold wedge as above indicates that the term belongs to the $i$-th tensor factor in $\mathcal{P}_k$.  Since $\varphi_1$ is closed and $d$ satisfies \eqref{productrule}, it follows that $\varphi_k$ is also closed.  The cocycle $\varphi_k$ and its analogues for $\SO(p,q)$ and $\mathrm{SU}(p,q)$ played a key role in the work of Kudla and Millson.

\section{Some decomposition results}

For $1 \leq i,j \leq k$, we define the following second order partial differential operator acting on $\mathrm{Pol}((V_+ \otimes \C)^k)$
$$\Delta_{ij} = \sum_{\alpha=1}^n  \frac{\partial^2}{\partial z_{\alpha,i} \partial z_{\alpha,j}}.$$

We define $\mathcal{H}((V_+ \otimes \C)^k)$ to be the subspace of harmonic (annihilated by all $\Delta_{ij}$) polynomials.  Then we have the classical result, see \cite{KV}, Lemma 5.3, (note that we have reversed their $n$ and $k$)
\begin{thm}\label{generalcase}
\begin{enumerate}
\item The map 
$$p(r_{11},r_{12}, \ldots, r_{kk}) \otimes h(z_{11},z_{12},\ldots,z_{nk})\to p(r_{11},r_{12}, \ldots, r_{kk})h(z_{11},z_{12},\ldots,z_{nk})$$
induces a surjection
$$\mathcal{R}_k(V_+ \otimes \C) \otimes \mathcal{H}((V_+ \otimes \C)^k) \to \mathrm{Pol}((V_+ \otimes \C)^k).$$
\item In case $2k < n$ the map is an isomorphism.
\end{enumerate}
\end{thm}
\begin{rmk}
We will denote this surjection by writing 
$$\mathrm{Pol}((V_+ \otimes \C)^k) = \mathcal{R}_k(V_+ \otimes \C) \cdot \mathcal{H}((V_+ \otimes \C)^k).$$

\end{rmk}

We will need the following three decomposition results.  First, recall $V_+$ is the span of $e_1,\ldots,e_n$.  Then we have
\begin{lem}\label{Minsum}
$$\mathcal{P}_k = \mathrm{Pol}((V_+ \otimes \C)^k) \otimes \C[w_1,\ldots,w_k]$$
\end{lem}
and

\begin{lem}\label{posharm}
$$\mathrm{Pol}((V_+ \otimes \C)^k) = \C[r_{11},r_{12}, \ldots, r_{kk}] \cdot  \mathcal{H}((V_+ \otimes \C)^k).$$
\end{lem}

We make the following definition

\begin{defn}
$$\mathcal{S}_k = \mathcal{P}_k^{\OO(n, \C)} = \C[r_{11},r_{12}, \ldots, r_{kk},w_1, \ldots,w_k].$$
\end{defn}
Then we have

\begin{lem}\label{combination}
$$\mathcal{P}_k = \mathcal{S}_k \cdot \mathcal{H}((V_+ \otimes \C)^k).$$
\end{lem}
It will be very  useful to us  that as $\SO(n)$-modules we have

\begin{equation}
\mathfrak{p} \cong \mathfrak{p}^* \cong V_+ \otimes \C
\end{equation}
and hence

\begin{equation}\label{gothicp}
\wwedge{i}(\mathfrak{p}^*) \cong \wwedge{i}( V_+ \otimes \C).
\end{equation}

\section{The occurrence of the $\OO(n,\C)$-module $\wwedge{\ell}(V_+ \otimes \C)$ in \\
 $\mathcal{H} \big( (V_+ \otimes \C)^k \big)$} \label{realrep}

In this section we compute the $\wwedge{\ell}(V_+ \otimes \C)$ isotypic subspaces for $\OO(n, \C)$ acting on $\mathcal{H} \big( (V_+ \otimes \C)^k \big)$ where we identify $V_+ \otimes \C$ with $\C^n$ and hence $\OO(V_+ \otimes \C)$ with $\OO(n, \C)$ using the basis $e_1, \ldots, e_n$.

 It is standard to parametrize the irreducible representations of $\OO(n, \C)$  by Young diagrams such that the sum of the lengths of the first two columns is less than or equal to $n$, see \cite{Weyl}, Chapter 7, \S 7.  In \cite{Howe}, Howe defines the depth of an irreducible representation to be the number of rows in the associated diagram.

\begin{lem} \label{vanext}
$$\Hom_{\OO(n, \C)} \big( \wwedge{\ell}(V_+ \otimes \C), \mathcal{H} \big( (V_+ \otimes \C)^k \big) \big) \neq 0 \text{ if and only } \ell \leq k.$$
\end{lem}

\begin{proof}
We will use Proposition 3.6.3 in \cite{Howe}.  $\wwedge{\ell}(V_+ \otimes \C)$ corresponds to the diagram $D$ which is a single column of length $\ell$.  Hence, $\wwedge{\ell}(V_+ \otimes \C)$ has depth $\ell$.  But, Proposition 3.6.3 states that a representation of depth $\ell$ occurs in $\mathcal{H} \big( (V_+ \otimes \C)^k \big)$ if and only if $\ell \leq k$.

\end{proof}

\begin{rmk}
The lemma also follows from Proposition 6.6 ($n$ odd), and  Proposition 6.11 ($n$ even) of Kashiwara-Vergne, \cite{KV}.
\end{rmk}

\begin{lem} \label{intertwine}
\hfill
\begin{enumerate}
\item If $U_1$ is an irreducible representation of $\OO(n, \C)$, then
$$U_2 = \Hom_{\OO(n, \C)} \big( U_1, \mathcal{H} \big( (V_+ \otimes \C)^k \big) \big)$$
is an irreducible representation of $\GL(k, \C)$.
\item Hence, given $U_1$ as above, there exists a unique representation $U_2$ of \\
$\GL(k, \C)$ such that we have an $\OO(n, \C) \times \GL(k, \C)$ equivariant embedding
$$\Psi : U_1 \otimes U_2 \to \mathcal{H} \big( (V_+ \otimes \C)^k \big).$$
\end{enumerate}
\end{lem}

\begin{proof}
Statement (1) follows from Proposition 5.7 of Kashiwara-Vergne \cite{KV}.\\
Statement (2) follows from statement (1).

\end{proof}

\begin{lem} \label{corext}
In the set-up of the previous lemma, if $U_1$ is $\wwedge{\ell}(V_+ \otimes \C)$, then $U_2$ is $\wwedge{\ell}(\C^k)$ and hence the $\wwedge{\ell} (V_+ \otimes \C)$ isotypic subspace for $\OO(n, \C)$ in $\mathcal{H} \big( (V_+ \otimes \C)^k \big)$ is $\wwedge{\ell}(V_+ \otimes \C) \otimes \wwedge{\ell}(\C^k)$ as an $\OO(n, \C) \times \GL(k, \C)$-module.
\end{lem}

\begin{proof}
This is an immediate consequence of Howe \cite{Howe} Proposition 3.6.3 or \\
Kashiwara-Vergne \cite{KV} Theorem 6.9 ($n$ odd) and Theorem 6.13 ($n$ even).

\end{proof}

In the next section we construct an explicit $\OO(n, \C) \times \GL(k,\C)$-intertwiner
$$\Psi: \wwedge{\ell}(V_+ \otimes \C) \otimes \wwedge{\ell} (\C^k) \to \mathcal{H} \big( (V_+ \otimes \C)^k \big).$$  Thus we will give a direct proof of Lemma \ref{corext} and  the ``if'' part of  Lemma \ref{vanext}. 

\section{The intertwiner $\Psi$}
In what follows we will identify the space $\mathrm{Hom}(\C^k, V_+ \otimes \C)$ with the space of $n$ by $k$ matrices using the basis $e_1,\ldots, e_n$ for $V_+$.  That is,
$$\mathrm{Hom}(\C^k, V_+ \otimes \C) \cong (\C^k)^* \otimes (V_+ \otimes \C) \cong (V_+ \otimes \C)^k \cong M_{n,k}(\C).$$
Let $\ell \leq k$.  Let $J = (j_1,\ldots, j_\ell)$ be in $\mathcal{S}_{\ell,k}$ and $I =(i_1,i_2, \ldots,i_{\ell})$ be in $\mathcal{S}_{\ell,n}$, that is, strictly increasing $\ell$-tuples of elements of $\{1, \ldots, k\}$ and $\{1, \ldots, n\}$ respectively.  Let $\epsilon_1,\ldots, \epsilon_k$ be the standard basis for $\C^k$ and $\alpha_1,\ldots, \alpha_k$ be the dual basis.  Let $\{e_I^*\}$ be the basis of $\wwedge{\ell}(V_+ \otimes \C)^*$ dual to the basis $\{e_I\}$ of $\wwedge{\ell}(V_+ \otimes \C)$.  Now define $e_I$ and $\epsilon_{J}$ by
$$e_I = e_{i_1} \wedge \cdots \wedge e_{i_\ell} \quad \text{and} \quad \epsilon_{J} = \epsilon_{j_1} \wedge \cdots  \epsilon_{j_\ell}.$$
Then we define the intertwiner
\begin{align*}
\Psi: \wwedge{\ell}(V_+ \otimes \C) \otimes \wwedge{\ell} \C^k &\rightarrow \mathrm{Pol}\big( \Hom(\C^k, V_+ \otimes \C) \big) \\
\Psi(e_I \otimes \epsilon_J)(Z) &= e_I^* \big(\wwedge{\ell}(Z) (\epsilon_J) \big).
\end{align*}
Here $Z \in \mathrm{Hom}(\C^k, V_+ \otimes \C)$ and $\wwedge{\ell}(Z) : \wwedge{\ell}(\C^k) \to \wwedge{\ell}(V_+ \otimes \C)$ is the $\ell^{th}$ exterior power of $Z$.  Clearly $\Psi$ is nonzero.

Now define $f_{I,J}(Z)$ to be the determinant of the $\ell$ by $\ell$ minor given by choosing the rows $i_1,i_2,\ldots,i_{\ell}$ and the columns $j_1,\ldots, j_\ell$ of $Z \in \mathrm{Hom}(\C^k, V_+ \otimes \C)$ regarded as an $n$ by $k$ matrix. 
Then we have
\begin{lem}
$$\Psi(e_I \otimes \epsilon_J)(Z) = f_{I,J}(Z).$$
\end{lem}

\begin{lem}
$f_{I,J}(Z)$ is harmonic.
\end{lem}

\begin{proof}
Given a monomial $m$, we have $\Delta_{ij}(m) \neq 0$ if and only if $m = z_{\alpha, i} z_{\alpha, j} m^\prime$ for some $1 \leq \alpha \leq n$ and non-zero monomial $m^\prime$.  But since $f_{I,J}$ is a determinant, it is the sum of monomials each of which has at most one term from a given row.  That is,
$$\frac{\partial^2}{\partial z_{\alpha, i} \partial z_{\alpha, j}} f_{I,J}(Z) = 0$$
for all $i,j, \alpha$ and thus $\Delta_{ij}\big( f_{I,J}(Z) \big) = 0$ term by term.

\end{proof}

\begin{cor}
The intertwiner $\Psi$ maps to the harmonics, that is
$$\Psi: \wwedge{\ell}(V_+ \otimes \C) \otimes \wwedge{\ell} \C^k \rightarrow \mathcal{H}\big( (V_+ \otimes \C)^k \big).$$
\end{cor}

We leave the following proof to the reader.

\begin{lem}
$\Psi$ is $\OO(n, \C) \times \GL(k, \C)$-equivariant.  That is, for $g \in \OO(n, \C)$ and $g^\prime \in \GL(k, \C)$,
\begin{equation}\label{equivariant} 
\Psi \big( \wwedge{\ell}(g)(e_I) \otimes \wwedge{\ell}(g^\prime)(\epsilon_J) \big)(Z) = \Psi(e_I \otimes \epsilon_J)(g^{-1} Z g^\prime).
\end{equation}
\end{lem}

We note that $\Psi$ is equivalent to a bilinear map $\widetilde{\Psi}$ from  $\wwedge{\ell}(V_+ \otimes \C) \times \wwedge{\ell} \C^k$ to $\mathcal{H} \big( (V_+ \otimes \C)^k \big)$ where $\widetilde{\Psi}(\eta, \tau) = \Psi(\eta \otimes \tau)$.  We define
\begin{equation*}
\widetilde{\Psi}^\prime : \wwedge{\ell}(\C^k) \to \mathrm{Hom}_{\OO(n, \C)} \big( \wwedge{\ell}(V_+ \otimes \C), \mathcal{H} \big( (V_+ \otimes \C)^k \big) \big)
\end{equation*}
by
\begin{equation}\label{formulaforpsi}
\widetilde{\Psi}^\prime (\epsilon_J) = \widetilde{\Psi} (\bullet, \epsilon_J) = \sum_{I \in \mathcal{S}_{\ell,n} } f_{I,J} e_I^*.
\end{equation}
We define
\begin{equation*}
\Psi_J = \widetilde{\Psi}^\prime (\epsilon_J) \in \Hom_{\OO(n,\C)}\big( \wwedge{\ell}(V_+ \otimes \C), \mathcal{H} \big( (V_+ \otimes \C)^k \big) \big)
\end{equation*}
and thus
$$\Psi_J(Z) = \sum_{I \in \mathcal{S}_{\ell,n}} f_{I,J}(Z) e_I^*.$$

We now compute $\Hom_{\OO(n,\C)}\big( \wwedge{\ell}(V_+ \otimes \C), \mathcal{H} \big( (V_+ \otimes \C)^k \big) \big)$. Since $\wwedge{\ell} (V_+ \otimes \C) \otimes \wwedge{\ell} (\C^k)$ is an irreducible $\OO(n, \C) \times \GL(k,\C)$-module it follows that $\Psi$ is injective. The image of $\Psi$ is contained in the $\wwedge{\ell} (V_+ \otimes \C)$ isotypic subspace of $\mathcal{H}(V^k_+)$.  By Lemma \ref{corext}, we know that the $\wwedge{\ell}(V_+ \otimes \C)$-isotypic subspace is isomorphic to this tensor product as an $\OO(n, \C) \times \GL(k,\C)$-module which is irreducible.  Hence $\Psi$ is a nonzero map of irreducible $\OO(n, \C) \times \GL(k, \C)$-modules and hence an isomorphism.  Thus $\widetilde{\Psi}^\prime$ is an isomorphism of $\C$-vector spaces and we have

\begin{lem} \label{psijbasisharm}
$\{\Psi_J : J \in \mathcal{S}_{m,k}\}$ is a basis for the vector space
$$\Hom_{\OO(n,\C)}\big( \wwedge{\ell}(V_+ \otimes \C), \mathcal{H} \big( (V_+ \otimes \C)^k \big) \big).$$
\end{lem}

\begin{prop} \label{Psi_Jbasis}
$\{\Psi_J : J \in \mathcal{S}_{\ell,k}\}$ is a basis for the $\mathcal{S}_k$-module
$$\Hom_{\OO(n,\C)}\big( \wwedge{\ell}(V_+ \otimes \C), \mathcal{P}_k \big).$$
\end{prop}

\begin{proof}
By Lemma \ref{psijbasisharm}, we have $\{\Psi_J : J \in \mathcal{S}_{\ell,k}\}$ is a basis for the $\C$-vector space\\
$\Hom_{\OO(n, \C)} \big( \wwedge{\ell}(V_+ \otimes \C), \mathcal{H} \big( (V_+ \otimes \C)^k \big) \big)$.  Since $\Hom_{\OO(n, \C)}(\C, \cdot)$ is exact, the surjection $\mathcal{S}_k \otimes \mathcal{H} \big( (V_+ \otimes \C)^k \big) \to \mathcal{P}_k$ induces a surjection $\phi$ from the $\C$-vector space $\mathcal{S}_k \otimes \Hom_{\OO(n, \C)} \big( \wwedge{\ell}(V_+ \otimes \C), \mathcal{H} \big( (V_+ \otimes \C)^k \big) \big)$ to the $\C$-vector space $\Hom_{\OO(n, \C)} \big( \wwedge{\ell}(V_+ \otimes \C), \mathcal{P}_k \big)$, given by $\phi(f \otimes T) = fT$ and thus \\
$\{\Psi_J : J \in \mathcal{S}_{\ell,k}\}$ spans $\Hom_{\OO(n, \C)} \big( \wwedge{\ell}(V_+ \otimes \C), \mathcal{P}_k \big)$ as an $\mathcal{S}_k$-module.

We now show the elements of this set are independent over $\mathcal{S}_k$.  We claim that if $Z_0 \in \Hom^0(\C^k, V_+ \otimes \C)$, the set of injective homomorphisms, then $\{\Psi_J(Z_0)\}$ is an independent set over $\C$.  Indeed, $\Psi_J(Z_0) = \wwedge{\ell}(Z_0) \epsilon_J \in \wwedge{\ell}(V_+ \otimes \C)$.  And, since $Z_0$ is an injection, so is $\wwedge{\ell}(Z_0)$.  Thus, since $\{\epsilon_J\}$ is an independent set over $\C$ and $\wwedge{m}(Z_0)$ is an injection, $\{\wwedge{\ell}(Z_0) \epsilon_J\}$ is an independent set over $\C$.

Following Equation \eqref{r_{ij}} we have the quadratic $\OO(n, \C)$-invariants $r_{ij}(Z)$ for $Z \in \Hom(\C^k, V_+ \otimes \C)$, where $r_{ij}(Z) = \big( Z(\epsilon_i), Z(\epsilon_j) \big)$ is the inner product of columns.  We will regard $r_{ij}$ both as matrix indeterminates and functions of $X$.  Recall $\mathcal{S}_k = \C[r_{11}, r_{12}, \ldots, r_{kk}, w_1, \ldots, w_k]$.

Now suppose there is some dependence relation over $\mathcal{S}_k$ where $\mathbf{t} \in \C^k$ and we abbreviate $\mathbf{r} = (r_{11}, r_{12}, \ldots, r_{kk})$
$$\sum_J p_J(\mathbf{r}(Z), \mathbf{t}) \Psi_J(Z) = 0.$$
Then for each $(Z_0, \mathbf{t}_0) \in \Hom^0(\C^k, V_+) \times \C^k$, since $\{\Psi_J(Z_0, \mathbf{t}_0)\}$ is independent over $\C$, we have that $p_J(\mathbf{r}(Z_0), \mathbf{t}_0) = 0$ for all $J$ .  And since $k \leq n$, $\Hom^0(\C^k, V_+) \times \C^k$ is dense in $\Hom(\C^k, V_+) \times \C^k$, and thus we have $p_J = 0$ for all $J$.

\end{proof}

\section{Computation of the spaces of cochains}
Recall that $\OO(n) \subset \OO(n,1)$ is the subgroup that fixes the last basis vector $e_{n+1}$.

Recall $\mathcal{I}_{a,n}$ was defined to be the set of all ordered $a$-tuples of distinct elements from $\{1, \ldots, n\}$, equivalently the set of all injective maps from $\{1, \ldots, a \}$ to $\{1, \ldots, n\}$.  We recall $\mathcal{S}_{a, n} \subset \mathcal{I}_{a.n}$ is the set of all strictly increasing $a$-tuples.  Lastly, given $I = (i_1, \ldots, i_a) \in \mathcal{I}_{a,n}$, we define the set $\overline{I} = \{i_1, \ldots, i_a\} \subset \{1, \ldots, n\}$.  Note that this map restricted to $\mathcal{S}_{a,n}$ is a bijection to its image, the set of all subsets of size $a$ of $\{1, \ldots, n\}$.

Let $\ast: \wwedge{\ell} \mathfrak{p}_0^{*} \to \wwedge{n - \ell} \mathfrak{p}_0^{*}$ be the Hodge star operator associated to the Riemannian metric and the volume form $\vol = \omega_1 \wedge \cdots \wedge \omega_n$.  Extend $*$ to $\wwedge{\ell} \mathfrak{p}^*$ to be complex linear.  Hence
\begin{equation} \label{starg}
* \circ g = (\det g) g \circ * \text{ for } g \in \OO(n, \C).
\end{equation}

We now observe that the complex $C = C^{\bullet}(\mathfrak{so}(n,1, \C) ,\SO(n, \C); \mathcal{P}_k)$ is the direct sum of two subcomplexes $C_+$ and $C_-$.  Indeed let $\iota \in \mathrm{O}(n, \C)$ be the element satisfying 
\begin{equation} \label{iotadef}
\iota(e_1) = -e_1 \ \text{and} \ \iota(e_j) = e_j, \quad 1 < j \leq n+1.
\end{equation}

Then $\iota \otimes \iota$ acts on the complex $C$ and commutes with $d$.  We define $C_+$ resp. $C_-$ to be the $+1$ resp. $-1$ eigenspace of $\iota \otimes \iota$.  Then we have
$$C = C_+ \oplus C_-.$$
Hence, $H^\bullet(C) = H^\bullet(C_+) \oplus H^\bullet(C_-)$.  By Equation \eqref{starg}, $* \otimes 1$ anticommutes with $\iota \otimes \iota$ and hence 
$$C^{n-\ell}_- = (* \otimes 1) \big(C^\ell_+\big) .$$
Since $C^\ell = C_+^\ell \oplus C_-^\ell$, to compute $C^\ell$ it suffices to compute $C_+^\ell$ and $C_+^{n-\ell}$.  Hence it suffices to compute $C_+$.  We note
\begin{align*}
C_+^\ell &= C^{\ell}\big(\mathfrak{so}(n,1, \C),\SO(n, \C);\mathcal{P}_k \big)^{\iota \otimes \iota} \\
&= C^{\ell}\big(\mathfrak{so}(n,1, \C),\OO(n, \C);\mathcal{P}_k \big).
\end{align*}

\subsection{The computation of $C_+$}

We recall $\mathcal{R}_k = \C[r_{11},r_{12},\ldots, r_{kk}]$ and \\
$\mathcal{S}_k =  \mathcal{R}_k[w_1,\ldots,w_k] =  \C[r_{11},r_{12},\ldots, r_{kk},w_1,\ldots,w_k]$.

We define an isomorphism of $\mathcal{S}_k$-modules
\begin{align*}
F_\ell : \Hom_{\OO(n, \C)}(\wwedge{\ell}(V_+ \otimes \C), \mathcal{P}_k) &\to \Hom_{\OO(n, \C)}(\wwedge{\ell}\mathfrak{p}, \mathcal{P}_k) \\
e_I^* \otimes p &\mapsto \omega_I \otimes p
\end{align*}
and define, for $J \in \mathcal{S}_{\ell,n}$, $\varPhi_J$ by $F_\ell(\Psi_J) = \varPhi_J$.  Hence
$$\varPhi_J = \sum_{I \in \mathcal{S}_{\ell,n}} \omega_I \otimes f_{I,J}.$$
Then by Lemma \ref{vanext} we have

\begin{lem}
$C_+^\ell = 0$ for $\ell > k$.
\end{lem}

We then have the following consequence of Proposition \ref{Psi_Jbasis}
\begin{prop} \label{varphibasis}
$\{\varPhi_J\}$ is a basis for the $\mathcal{S}_k$-module $C_+^\ell$.
\end{prop}

We now give another description of $\varPhi_J$ as the outer exterior product of $\ell$ copies of $\varphi_1$.  That is,

\begin{equation} \label{varphiJwedgeproduct}
\varphi_{1}^{(j_1)} \wedge \cdots \wedge \varphi_{1}^{(j_\ell)} = \sum_{I \in \mathcal{S}_{\ell,n}} \omega_I  \otimes f_{I,J} = \varPhi_J.
\end{equation}

\begin{rmk}
$$C_+^k =  \mathcal{S}_k \varphi_k$$
\end{rmk}
where $\varphi_k = \varPhi_{1,2, \ldots, k}$ as before.

\subsection{The computation of $C_-$}\label{highdegreecochains}

Since $* \otimes I$ is an isomorphism of $\mathcal{S}_k$-modules from $C_+^\ell \to C_-^{n-\ell}$, we have, abbreviating $(* \otimes I)(\varPhi_J)$ to $(*\varPhi_J)$,

\begin{prop} \label{thecochains}
$$C_+^\ell = \sum_{J \in \mathcal{S}_{\ell, k}} \mathcal{S}_k \varPhi_J$$
$$C_-^\ell = \sum_{J \in \mathcal{S}_{n-\ell, k}} \mathcal{S}_k (*\varPhi_J)$$
where if $i > j$ then $\mathcal{S}_{i,j}$ is the empty set.  In particular $C_+^\ell \cong \mathcal{S}_k^{\binom{k}{\ell}}$ for $0 \leq \ell \leq k$ and zero for $\ell > k$, whereas $C_-^\ell \cong \mathcal{S}_k^{\binom{k}{n-\ell}}$ for $\ell \geq n-k$ and zero for $\ell < n-k$.
\end{prop}

\section{The computation of the cohomology of $C_+$} \label{computationofcplus}

We will compute the cohomology of the associated graded complex $E_0 = \mathrm{gr}(C)$.  By Remark \ref{d=d_{0,1}}, our differential is $d_{0,1}$.  We abuse notation and call this operator $d$.  We will see there is only one non-zero cohomology group and use the results of Section \ref{spectralsection} to compute the cohomology of the original complex.

Throughout this section, $J$ will denote an element of $\mathcal{S}_{\ell,n}$.  To simplify the notation in what follows, we will abbreviate $\omega \otimes \varphi$ to $\varphi \omega$ for $\omega \in \wwedge{\bullet} \mathfrak{p}^*$ and $\varphi \in \mathcal{P}_k$.  In particular,

\begin{equation} \label{defofphij}
\varPhi_J= \sum_{I \in \mathcal{S}_{\ell,n}} f_{I,J} \omega_I.
\end{equation}

Recall we have fixed $k$ with $k < n$ and we have the ring
\begin{equation}
\mathcal{S}_k = \C[r_{11},r_{12},\ldots, r_{kk},w_1,\ldots,w_k].
\end{equation}

For $1 \leq i \leq k, J \in \mathcal{S}_{\ell,k}$, if $i \notin \overline{J}$, then we denote by $\{J, i\}$ the element of $\mathcal{S}_{\ell+1, k}$ that corresponds to the set $\overline{J} \cup \{i\}$.  Now we define $\varPhi_{J, i} \in C_+^{\ell+1}$  by

\begin{equation}
\varPhi_{J,i} = \begin{cases} (-1)^{J(i)} \varPhi_{\{J, i\}} & \text{ if } i \notin \overline{J} \\ 0 & \text{ if } i \in \overline{J}.
\end{cases}
\end{equation}
where $J(i)$ is defined as follows.

\begin{defn}
$J(i) = |\{j \in \overline{J} : j < i\}|$ is the number of elements of $J$ less than $i$.
\end{defn}

We remark that the reason for the sign $(-1)^{J(i)}$ in this notation is that we have put the $i$ in the appropriate spot instead of the beginning.  In particular, we have the following lemma.

\begin{lem}
$$\varphi_1^{(i)} \wedge \varPhi_J = \varPhi_{J,i}.$$
\end{lem}

The following formula for $d$ is then immediate.

\begin{prop} \label{donbasisfirstcase}
\begin{equation}
d \varPhi_J =   \sum_{i=1}^k w_i \varphi_1^{(i)} \wedge \varPhi_J = \sum_{i=1}^k w_i \varPhi_{J,i}
\end{equation}
\end{prop}

\subsection{The map from $\mathrm{gr}(C_+)$ to a Koszul complex $K_+$}

We define $K_+$ to be the complex given by
\begin{equation}
K_+^\bullet = \wwedge{\bullet}((\C^k)^*) \otimes \mathcal{S}_k \text{ with the differential } d_{K_+} = \sum_i A(dw_i) \otimes w_i.
\end{equation}
Here $w_1, \ldots, w_k$ are coordinates on $\C^k$ and $dw_1, \ldots, dw_k$ are the corresponding one-forms.

We define a map $\Psi_+$ of $\mathcal{S}_k$-modules from the associated graded complex $\mathrm{gr}(C_+)$ to $K_+$ by sending $\varPhi_J$ to $dw_J$.  In particular, this sends $\varphi_1^{(i)} \mapsto dw_i$.  Recall that the degrees $\ell$ such that $C_+^\ell$ is non-zero range from $0$ to $k$.

The following lemma is an immediate consequence of Proposition \ref{donbasisfirstcase}. We leave its verification to the reader.

\begin{lem} $\Psi_+$  is an isomorphism  of cochain complexes, $\Psi_+ : \mathrm{gr}(C_+) \to K_+$.
\end{lem} 

We now compute the cohomology of the complex $K_+, d_{K_+}$. 

\begin{prop}\label{K+vanishing}
\hfill

\begin{enumerate}
\item $H^\ell(K_+) = 0,  \ell \neq k$ 
\item $H^k(K_+) = \mathcal{S}_k / (w_1, \ldots, w_k) dw_1 \wedge \cdots \wedge dw_k \cong \mathcal{R}_k dw_1 \wedge \cdots \wedge dw_k$.
\end{enumerate}
\end{prop}

We first prove statement $(1)$ of Proposition \ref{K+vanishing}.  We first note that 
$$ d_K = \sum_j w_j \otimes A(dw_j)$$
is the differential in the Koszul complex $K_+(w_1,\ldots,w_k)$ associated to the sequence of the linear polynomials $w_1,\ldots,w_k$, see Eisenbud \cite{Eisenbud}, Section 17.2.  To see that the Koszul complex as described in \cite{Eisenbud} is the above complex $K$ we choose $\mathcal{S}_k$ as Eisenbud's ring $R$ and $\mathcal{S}_k^k$ as Eisenbud's module $N$. In our description we are using the exterior algebra $\bigwedge^{\bullet}( (\C^k)^*) \otimes \mathcal{S}_k$.  But the operation of taking the exterior algebra of a module commutes with base change and hence we have $\bigwedge^{\bullet}( (\C^k)^*) \otimes \mathcal{S}_k \cong \bigwedge^{\bullet}(\mathcal{S}_k^k)$. Then we apply Eisenbud's construction with the sequence $w_1,\ldots, w_k$ to obtain the above complex $K_+$.  We recall that $f_1, \ldots, f_k$ is a regular sequence in a ring $R$ if and only if $f_i$ is not a zero divisor in $R/(f_1, \ldots, f_{i-1})$ for $1 \leq i \leq k$.  The following lemma is obvious.

\begin{lem} \label{K+regularsequence}
$w_1, \ldots, w_k$ is a regular sequence in $\mathcal{S}_k$.
\end{lem}

Statement $(1)$ of Proposition \ref{K+vanishing} then follows from Lemma \ref{K+regularsequence} above and Corollary 17.5 of \cite{Eisenbud} (with $M = R$), which states that the cohomology of a Koszul complex $K(f_1,\ldots,f_k)$ below the top degree vanishes if $f_1,\ldots,f_k$ is a regular sequence.

We next note that statement $(2)$ of Proposition \ref{K+vanishing} follows from \cite{Eisenbud}, Corollary 17.5, which states that if $f_1,\ldots,f_k$ is a regular sequence in the ring $R$ then the top cohomology $H^k(K(f_1, \ldots, f_k))$ is isomorphic to $R/(f_1, \ldots, f_k)$.

We now pass from the above results for $K_+$ to the corresponding results for $C_+$.  

\begin{thm}\label{K+vanCvan}
\hfill

\begin{enumerate}
\item $H^{\ell}(C_+) = 0$ $\ell \neq k$
\item $H^k(C_+) = \mathcal{S}_k / (w_1, \ldots, w_k) \varPhi_k \cong \mathcal{R}_k \varphi_k$.
\end{enumerate}
\end{thm}

\begin{proof}
Since $K_+$ is the associated graded complex of $C_+$, the statement (1) is an immediate consequence of Proposition \ref{K+vanishing} and Proposition \ref{grCzeroimpliesCzero}.

Statement $(2)$ follows by applying statement $(3)$ of Proposition \ref{generalspectral} and noting that $\varPhi_k$ is the form $\varphi_k$ of Kudla and Millson.

\end{proof}

\section{The computation of the cohomology of $C_-$} \label{computationofcminus}
We now compute the cohomology of the associated graded complex $\mathrm{gr}(C_-)$.  As in the previous section, the differential is $d_{0,1}$.  Again, we abuse notation and call this operator $d$.  We will see there is only one non-zero cohomology group and use the results of Section \ref{spectralsection} to compute the cohomology of the original complex.

In Proposition \ref{thecochains} we proved that $\{* \Phi_J : J \in \mathcal{S}_{n-\ell, k} \}$ was a basis for the $\mathcal{S}_k$-module $C_-^\ell$.  Note that in order to obtain a cochain of degree $\ell$, we assume $J \in \mathcal{S}_{n-\ell, k}$ instead of $\mathcal{S}_{\ell,k}$, since by Equation \eqref{defofphij}, for $J \in \mathcal{S}_{\ell, k}, \varPhi_J= \displaystyle \sum_{I \in \mathcal{S}_{\ell,n}} f_{I,J} \omega_I$ and hence

\begin{equation} \label{def*phij}
*\varPhi_J =   \sum_{I \in \mathcal{S}_{\ell,n}}   f_{I,J} (*\omega_I)
\end{equation}
has degree $n-\ell$.  For our later computations, we need to replace the determinant $f_{I,J}$ of \eqref{def*phij} by the monomials $z_{I,J}$ where $z_{I,J} = z_{i_1,j_1} \cdots z_{i_{n-\ell},j_{n-\ell}}$.  In order to do this, we sum over all ordered subsets $\mathcal{I}_{n-\ell, n}$, instead of just $\mathcal{S}_{n-\ell, n}$ (those which are in increasing order), to obtain

\begin{equation} \label{defof*phij}
*\varPhi_J =  \sum_{I \in \mathcal{I}_{n-\ell,n}}   z_{I,J} (*\omega_I).
\end{equation}

Using this basis we may identify $C_-^\ell$ with the direct sum of $\binom{k}{n-\ell}$ copies of $\mathcal{S}_k$.

\subsection{A formula for $d$} \label{dformula}
\hfill

Our goal is to prove Proposition \ref{donbasis}, a formula for $d$ relative to the basis $\{*\varPhi_J\}$.  Recall that for $J \in \mathcal{S}_{n-\ell,k}$, and $1 \leq i \leq k$, we have $J(i)$ is the number of elements in $J$ less than $i$.   For $1 \leq i \leq k$, $J \in \mathcal{S}_{\ell,k}$, if $i \in \overline{J}$, then we denote by $\{J - i\}$ the element of $\mathcal{S}_{\ell-1, k}$ that corresponds to the set $\overline{J} - \{i\}$.  Now we define $*\varPhi_{J - i} \in C_+^{\ell+1}$  by

\begin{equation}
\varPhi_{J-i} = \begin{cases} (-1)^{J(j)} \varPhi_{\{J - i\}} & \text{ if } i \in \overline{J} \\ 0 & \text{ if } i \notin \overline{J}.
\end{cases}
\end{equation}



\begin{prop} \label{donbasis}
Assume $|J| = n-\ell$, then we have
\begin{equation*}
d(*\varPhi_J) = (-1)^{(n-\ell-1)} \big( \sum_{j \in J}   \sum_{i =1}^k w_i r_{ij} *\varPhi_{J-j}    \big).
\end{equation*} 
\end{prop}
The proposition will follow from the next two lemmas. Note first that from the defining formula we have
\begin{equation}\label{firstformulaford}
d(*\varPhi_J) = \sum_{i = 1}^k \sum_{\alpha=1}^n z_{\alpha,i} w_i \omega_\alpha \wedge (*\varPhi_J).
\end{equation}

In what follows we will need to extend the definition of $J-j$ for $J \in \mathcal{S}_{n-\ell,k}$ to elements $I \in \mathcal{I}_{n-\ell,n}$.  Given $I \in \mathcal{I}_{n-\ell,n}$, we define the symbol $I-i_s$ to be the element $(i_1, \ldots, \widehat{i_s}, \ldots, i_{n-\ell}) \in \mathcal{I}_{n-\ell-1,n}$, the symbol $\widehat{i_s}$ indicating that the term $i_s$ is omitted.

We leave the proof of the next lemma to the reader.

\begin{lem} \label{wedgeandstar}
Given $I \in \mathcal{I}_{n-\ell,n}$,
\begin{align*}
\omega_{\alpha} \wedge *(\omega_I) = &(-1)^{(n-\ell-1)}* \iota_{e_{\alpha, n+1}}(\omega_I)\\
= & (-1)^{(n-\ell-1)}\sum_{s=1}^{n-\ell} (-1)^{s-1}\delta_{\alpha, i_{s}}  *  (\omega_{I-i_s}).\\
\end{align*}

Here $\delta_{\alpha, i_{s}}$ is the Kronecker delta 
$$\delta_{\alpha, i_{s}} = \begin{cases} 1 & \ \text{if} \  \alpha = i_{s}  \\ 0 & \ \text{if} \  \alpha  \neq i_{s} 
\end{cases}$$
\end{lem}

\begin{lem} \label{mainformulaford}
\begin{equation*}
\sum_{\alpha=1}^n z_{\alpha, i} A(\omega_\alpha) *\varPhi_J = \sum_{s=1}^{n-\ell} (-1)^{s-1} r_{i j_s} *\varPhi_{J-j_s}
\end{equation*}
\end{lem}

\begin{proof}
By Lemma \ref{wedgeandstar},
\begin{align*}
\sum_{\alpha=1}^n z_{\alpha, i} &A(\omega_\alpha) *\varPhi_J = (-1)^{n-\ell-1}* \sum_{\alpha=1}^n z_{\alpha, i} \iota_{e_{\alpha, n+1}} (\varPhi_J) \\
= &(-1)^{n-\ell-1}* \sum_{\alpha=1}^n z_{\alpha, i} \iota_{e_{\alpha, n+1}} ( \varphi_1^{(j_1)} \wedge \cdots \wedge \varphi_1^{(j_s)} \wedge \cdots \varphi_1^{(j_{n-\ell})} ) \\
= &(-1)^{n-\ell-1}* \sum_{s=1}^{n-\ell} (-1)^{s-1} \big( \varphi_1^{(j_1)} \wedge \cdots \wedge \sum_{\alpha = 1}^n z_{\alpha,i} \iota_{e_{\alpha, n+1}}(\varphi_1^{(j_s)}) \wedge \cdots \varphi_1^{(j_{n-\ell})} \big)
\end{align*}
where the second equality is by Equation \eqref{varphiJwedgeproduct}.  However,
\begin{align*}
\sum_{\alpha = 1}^n z_{\alpha,i} \iota_{e_{\alpha, n+1}}(\varphi_1^{(j_s)}) = & \sum_{\alpha = 1}^n z_{\alpha,i} \sum_{\beta=1}^n z_{\beta, j_s} \iota_{e_{\alpha, n+1}}(\omega_\beta) \\
= &\sum_{\alpha = 1}^n \sum_{\beta=1}^n z_{\alpha,i} z_{\beta, j_s}  \delta_{\alpha, \beta}\\
= &r_{i,j_s}.
\end{align*}
We see that in the $j_s^{th}$ slot we have replaced $\varphi_1$ by $(-1)^{s-1} r_{i,j_s}$ and the lemma follows.

\end{proof}

Proposition \ref{donbasis} follows by substituting the formula of Lemma \ref{mainformulaford} into Equation \eqref{firstformulaford}.

\subsection{The map from $\mathrm{gr}(C_-)$ to a Koszul complex $K_-$}

Define the cubic polynomials $c_j \in \mathcal{S}_k$ by
\begin{equation}\label{cubic}
c_j = \sum_{i=1}^k r_{ij} w_i, 1 \leq j \leq k.
\end{equation} 

We note that the $c_i$ are the result of the following matrix multiplication of elements of $\mathcal{S}_k$
\begin{equation*}
\begin{pmatrix}
r_{11} & \cdots & r_{1k} \\
\vdots & \ddots & \vdots \\
r_{k1} & \cdots & z_{kk}
\end{pmatrix}
\begin{pmatrix}
w_1 \\
\vdots \\
w_k
\end{pmatrix}
=
\begin{pmatrix}
c_1 \\
\vdots \\
c_k
\end{pmatrix}.
\end{equation*}
We then define $K_-$ to be the complex given by

\begin{equation} \label{newcomplexK-}
K_-^{\bullet} =  \wwedge{\bullet}(( \C^k)^*) \otimes \mathcal{S}_k  
\ \text{with the differential} \ d_{K_-} = \sum_{j=1}^k  A(dw_j) \otimes c_j.
\end{equation}

In order to obtain an isomorphism of complexes we need to shift degrees according to the following definition.

\begin{defn}
Let $M$ be a cochain complex and $j$ an integer.  Then we define the cochain complex $M[j]$ by $(M[j])^i = M^{j+i}$.
\end{defn}

We define a map $\Psi_-$ from the associated graded complex $\mathrm{gr}(C_-)[n-k]^{\ell}$ to $K_-^{\ell}$ by sending $* \varPhi_J$ to $* dw_J$.  Here by $*$ we mean the Hodge star for the standard Euclidean metric on $\R^k$ extended to be complex linear. Note that the degrees $\ell$ such that $C_-[n-k]^\ell$ is nonzero range from $0$ to $k$.

The following lemma is an immediate consequence of Proposition \ref{donbasis}. We leave its verification to the reader.

\begin{lem} $\Psi_-$  is an isomorphism  of cochain complexes 
$$\Psi_- : \mathrm{gr}(C_-)[n-k] \to K_-.$$
\end{lem} 

We now compute the cohomology of the complex $K_-, d_{K_-}$. 

\begin{prop}\label{Koszulvanishing}
\hfill

\begin{enumerate}
\item $H^\ell(K_-) = 0,  \ell \neq k$ 
\item $H^k(K_-) = \mathcal{S}_k / (c_1, \ldots, c_k) \vol$.
\end{enumerate}
\end{prop}

We mimic the proof of Proposition \ref{Avanishing}.  We note $d_{K_-}$ is the differential in the Koszul complex $K(c_1,\ldots,c_k)$ associated to the sequence of the cubic polynomials $c_1,\ldots,c_k$, \cite{Eisenbud}, Section 17.2.  It remains to show $\{c_j\}$ is a regular sequence.

\begin{lem} \label{regularsequence}
The sequence $(c_1, \ldots, c_k)$ is a regular sequence in $\mathcal{S}_k$.
\end{lem}

\begin{proof}
We follow the method of proof of Proposition \ref{Aregularsequence}.  By Lemma \ref{easyregular} we see that $(r_{11} w_1, \ldots, r_{kk} w_k)$ is a regular sequence in $\C[r_{11}, r_{22}, \ldots, r_{kk}, w_1, \ldots, w_k]$.  We now examine the sequence $\sigma = (\{r_{ij}\}_{i < j}, c_1, \ldots, c_k)$ of ``super-diagonal" $\{r_{ij}\}_{i<j}$ followed by $c_1, \ldots, c_k$.  That is,
\begin{equation*}
\sigma = (r_{12}, r_{13}, \ldots, r_{1k}, r_{23}, \ldots, r_{2k}, \ldots, r_{k-1,k}, c_1, \ldots, c_k).
\end{equation*}
It is clear that the ``super-diagonal" $r_{ij}$ form a regular sequence since they are coordinates.  To check if the $c_i$ are regular we work in
$$\mathcal{S}_k / (\{r_{ij}\}_{i < j}) \cong \C[r_{11}, r_{22}, \ldots, r_{kk}, c_1, \ldots, c_k].$$
The image of $c_i$ in this quotient ring is $r_{ii} w_i$ which form a regular sequence.

Now we apply Lemma \ref{Matsumuralemma} to reorder $\sigma$ and note that $(c_1, \ldots, c_k, \{r_{ij}\}_{i < j})$ is a regular sequence.  Hence $(c_1, \ldots, c_k)$ is a regular seqeuence.

\end{proof}

Proposition \ref{Koszulvanishing} then follows by the same argument that appears after Lemma \ref{K+regularsequence}.



We now pass from the above results for $K_-$ to the corresponding results for $C_-$.  

\begin{thm}\label{KvanCvan}
\hfill

\begin{enumerate}
\item $H^{\ell}(C_-) = 0$ $\ell \neq n$
\item $H^n(C_-) \cong \mathcal{S}_k / (c_1, \ldots, c_k)$.
\end{enumerate}
\end{thm}

\begin{proof}
Since $K_-$ is the associated graded complex of $C_-[n-k]$, statement (1) is a consequence of Proposition \ref{Koszulvanishing} and Proposition \ref{grCzeroimpliesCzero}.

Statement $(2)$ follows by applying statement $(3)$ of Proposition \ref{generalspectral}.

\end{proof}

\subsection{Infinite generation of $H^n(C)$ as an $\mathcal{R}_k$-module}

We will now demonstrate that $H^n(C)$ is not finitely generated as an $\mathcal{R}_k$-module.  In what follows, recall $\vol =\omega_1 \wedge \cdots \wedge \omega_n$.

\begin{prop}
The map from $\C[w_1, \ldots, w_k]$ to $H^n(C_-)$ sending $f$ to $[f \vol]$ is an injection.  Furthermore, $\C[w_1, \ldots, w_k]$ generates $H^n(C_-)$ over $\mathcal{R}_k$.
\end{prop}

\begin{proof}
There is an inclusion of polynomial algebras
$$\C[w_1, \ldots, w_k] \hookrightarrow \mathcal{S}_k.$$
This map has a right inverse, $\pi$, where $\pi(w_i) = w_i$ and $\pi(r_{ij})=0$.  Then since $\pi(c_j) = \pi(\sum_i r_{ij} w_i) = 0$, $\pi$ descends to a right inverse from $\mathcal{S}_k / (c_1, \ldots, c_k) \to \C[w_1, \ldots, w_k]$.  Hence the map is injective.

The second statement is obvious since $\C[w_1, \ldots, w_k]$ generates $\mathcal{S}_k$ as an $\mathcal{R}_k$-module.

\end{proof}

\begin{rmk}
Note that 
\begin{enumerate}
\item If $k \neq n$, then
$$H^n(C) = H^n(C_-) = \mathcal{S}_k / (c_1, \ldots, c_k) [\vol].$$
\item If $k=n$ then
$$H^n(C) = \mathcal{S}_n / (c_1, \ldots, c_n) [\vol] \oplus \mathcal{R}_n \varphi_n.$$
\end{enumerate}
\end{rmk}

\subsection{The decomposability of $H^n(C)$ as a $\mathfrak{sp}(2k, \R)$-module} Define $\iota^\prime \in O(n,1)$ by $\iota^\prime(e_j) = e_j, 1 \leq j \leq n$ and $\iota^\prime(e_{n+1}) = -e_{n+1}$.  Then since $\iota^\prime \otimes \iota^\prime$ acts on $\big( \wwedge{\ell}\mathfrak{p}^* \otimes \mathcal{P}_k \big)^{\SO(n)}$ and commutes with $d$, it acts on $H^n(C)$.   Since $\iota^\prime \otimes \iota^\prime$ has order two, we get the eigenspace decomposition into the $-1$ and $+1$ eigenspaces

$$H^n(C) = H^n(C)_- \oplus H^n(C)_+.$$

\begin{lem}
$H^n(C)_-$ and $H^n(C)_+$ are nonzero.
\end{lem}

\begin{proof}
Note that $\iota^\prime(\vol) = (-1)^n \vol$ and if $p(w_1, \ldots, w_k)$ is homogenous of degree $a$, then $\iota^\prime(p) = (-1)^a p$.  Hence $\iota^\prime \otimes \iota^\prime ([\vol \otimes p]) = (-1)^{n+a} [\vol \otimes p]$.

\end{proof}

Since the action of $\iota^\prime$ on $\mathcal{P}_k$ commutes with the action of $\mathfrak{sp}(2k, \R)$, the above decomposition of $H^n(C)$ is invariant under $\mathfrak{sp}(2k, \R)$.

\section{A simple proof of nonvanishing of $H^n(C)$.}\label{nonvanishingtop}
In what follows we let $G$ be a connected, noncompact, and semisimple Lie group with maximal compact $K$.  We let $n = \mathrm{dim}(G/K)$ and $\mathcal{V}$ be a  $(\mathfrak{g}, K)$-module with $\mathcal{V}^*$ the dual.
\begin{lem}
Suppose either

$(1)$ $\mathcal{V}$ is a topological vector space and $K$ acts continuously. Furthermore, assume there exists a nonzero $\mathfrak{g}$-invariant continuous linear functional $\alpha \in (\mathcal{V}^*)^{\mathfrak{g}}$\\
or

$(2)$ there exists a $\mathfrak{g}$-invariant linear functional $\alpha$ and a $K$-invariant vector $v \in \mathcal{V}$ such that $\alpha(v) \neq 0$ (no topological hypotheses needed).

\vspace{.25cm}

Then $H^n(\mathfrak{g}, K; \mathcal{V}) \neq 0$.
\end{lem}
\begin{proof}
We will first assume (2). We let $\vol$ be the element in $\wwedge{n}(\mathfrak{p}^*)$ which is of unit length for the metric induced by the Killing form and of the correct orientation. Then $\vol \otimes v$ is invariant under the product group $K \times K$, hence is invariant under the diagonal and hence gives rise to an $n$-cochain with values in $\mathcal{V}$ which is automatically a cocycle. Let $[\vol \otimes v]$ be the corresponding cohomology class. 
Now $\alpha$ induces a map on cohomology
\begin{equation*}
\alpha_* : H^n(\mathfrak{g}, K; \mathcal{V}) \to H^n(\mathfrak{g}, K; \R).
\end{equation*}
But $H^n(\mathfrak{g}, K; \R)$ is the ring of invariant differential $n$-forms on $D = G/K$.  Thus $H^n(\mathfrak{g}, K; \R) = \R [\vol]$.  Finally, we have
$$ \alpha_*[\vol \otimes v] = [\vol \otimes \alpha(v)] = \alpha(v) [\vol] \neq 0.$$
Hence, $[\vol \otimes v ] \neq 0$.

Now we reduce (1) to (2). 
 Since $\alpha \neq 0$ there is some $v \in \mathcal{V}$ such that $\alpha(v) \neq 0$. Let $dk$ be the Haar measure on $K$ normalized so that $\int\nolimits_{K} dk = 1$. We define the projection  $ p: \mathcal{V} \to \mathcal{V}^{K}$ by
\begin{equation*} \label{projection} 
 p(v) = \int_{K} k \cdot v dk.
\end{equation*}
The reader will verify since $\alpha$ is $K$ invariant and $\alpha(v)$ is continuous in $V$  that
\begin{equation*} \label{pinvariant}
 \alpha(p(v)) = \alpha(v).
\end{equation*}
Hence 
$$\alpha(p(v)) \neq 0 \ \text{and} \ p(v) \in \mathcal{V}^{K}.$$
Now the result follows from the argument of  case (2).

\end{proof}

We will apply $(2)$ for the following examples.  Let $G = \SO_0(p,q)$\\
(resp $\mathrm{SU}(p,q)$), $K = \SO(p) \times \SO(q)$ (resp $\mathrm{S}(\mathrm{U}(p) \times \mathrm{U}(q))$), $V = \R^{p,q}$ (resp $\C^{p,q}$) and $\mathcal{V} = \mathcal{S}(V^k)$, the space of Schwartz functions.  Then if $\alpha = \delta_0$, the Dirac delta distribution at the origin, $\alpha$ is a non-zero element of $(\mathcal{V}^*)^\mathfrak{g}$.  Hence, we have proved

\begin{thm} \label{topforweil}
For $\mathcal{V} = \mathcal{S}(V^k)$,
\begin{enumerate}
\item $H^{pq}(\mathfrak{so}(p,q), \SO(p) \times \SO(q); \mathcal{S}(V^k)) \neq 0$
\item $H^{2pq}(\mathfrak{u}(p,q), \mathrm{U}(p) \times \mathrm{U}(q); \mathcal{S}(V^k)) \neq 0$.
\end{enumerate}
\end{thm}


We give one more example which uses (1).Then (choosing $\alpha$ to be the Dirac delta function at the origin of $V$) we have
\begin{thm}\label{generaltheorem}
Let $G$ be a connected linear semisimple Lie group, $K$ a maximal compact subgroup and $n = \dim(G/K)$.  Let $V$ be a finite dimensional representation and $\mathcal{S}(V)$ be the Schwartz space of $V$.
$$H^n(\mathfrak{g}, K; \mathcal{S}(V)) \neq 0.$$
\end{thm}



\section{The extension of the theorem to the two-fold cover of $\mathrm{O}(n,1)$ }

\subsection{A general vanishing theorem in case the Weil representation is genuine} \label{genuinesection}

\begin{prop} \label{genuineprop}
Suppose $G$ is a semi simple subgroup of $\mathrm{Sp}(2N, \R)$. Let $\tilde{G} \rightarrow G$ be the pullback of the metaplectic extension of $\mathrm{Sp}(2N, \R)$ and $\tilde{K}$ be a maximal compact subgroup of $\tilde{G}$.  If some element of the center of $\tilde{G}$ acts by a multiple of the identity which is not one so, in particular, if the center of $\tilde{G}$ acts by a nontrivial character, then for all $\ell$
$$C^\ell(\mathfrak{g}, \tilde{K}; \mathcal{W}) = 0.$$
Hence the cohomology groups are all zero.  That is, for all $\ell$
$$H^\ell(\mathfrak{g}, \tilde{K}; \mathcal{W}) = 0.$$

\end{prop}

\begin{proof}

Let $Z(\tilde{G})$ be the center of $\tilde{G}$ and suppose $z \in Z(\tilde{G})$ acts by a nontrivial multiple of the identity.  Suppose $z$ generates the subgroup $A$ of $Z(\tilde{G})$.  Then $A \subset \tilde{K}$ and we have
$$\mathrm{Hom}_{\tilde{K}} (\wwedge{\ell} \mathfrak{p}^*, \mathcal{W}) \subset \mathrm{Hom}_{A} (\wwedge{\ell} \mathfrak{p}^*, \mathcal{W}).$$
But since $Z(\tilde{G})$ acts by conjugation on $\mathfrak{p}^*$, it acts trivially on $\mathfrak{p}^*$ and hence $A$ acts trivially on $\wwedge{\ell}\mathfrak{p}^*$.  But $z$ acts by a nontrivial multiple of the identity on $\mathcal{W}$.  Hence
$$\mathrm{Hom}_{A}(\wwedge{\ell} \mathfrak{p}^*, \mathcal{W}) = 0.$$

\end{proof}

\subsection{The computation for the two-fold cover of $\OO(n,1)$.}

In light of Proposition \ref{genuineprop}, we must twist the Weil representation by a character such that the center of $\widetilde{\OO(n,1)}$ acts trivially.

It is important to give the analogues of the results for $\SO_0(n,1)$ when we replace the connected group $\SO_0(n,1)$ by the covering group $\widetilde{\mathrm{O}(n,1)}$ (with four components) of $\OO(n,1)$ and hence the maximal compact $\SO(n)$ in $\SO_0(n,1)$ by the maximal compact subgroup $\widetilde{K} = \widetilde{\mathrm{O}(n) \times \mathrm{O}(1)}$ of $\widetilde{\mathrm{O}(n,1)}$.  Here  $\widetilde{\mathrm{O}(n,1)}$ denotes the total space of the restriction to $\mathrm{O}(n,1)$ of the pull-back of the metaplectic cover of $\Sp( 2k(n+1),\R )$ under the inclusion of the dual pair $\mathrm{O}(n,1) \times \Sp(2k,\R)$ into $\Sp( 2k(n+1),\R)$. Let $\varpi_k$ be the restriction of the Weil representation of $\mathrm{Mp}( 2k(n+1),\R)$ to $\widetilde{\mathrm{O}(n,1)}$ under the embedding $\widetilde{\mathrm{O}(n,1)} \to \mathrm{Mp} (2k(n+1),\R)$.  The following lemma is a consequence of the result of Section 4 of \cite{BMM2}). We believe it is more enlightening to state the following lemma in terms of a general orthogonal group.  Note that the required results for $\mathrm{O}(p,q)$ follow from those of \cite{BMM2} for $\mathrm{U}(p,q)$ by restriction. 

\begin{lem}\label{Weildescends}
\hfill

\begin{enumerate}
\item
The central extension $\widetilde{\mathrm{O}(p,q)} \to \mathrm{O}(p,q)$ is the pull-back under \\
$\det_{\OO(p,q)}^k: \OO(p,q) \to \C^{\ast}$ of the twofold extension $\C^{\ast} \to \C^{\ast}$ given by taking the square. Hence, the group $\widetilde{\mathrm{O}(p,q)}$, has the character  $\det_{\mathrm{O}(p,q)}^{k/2}$, the square-root of $\det_{\mathrm{O}(p,q)}^k$. 
\item
The character ${\det}_{\mathrm{O}(p,q)}^{k/2}$  is ``genuine'' (does not descend to the base of the cover)  if and  only if $k$ is odd.
\item For both even and odd $k$, the twisted Weil representation $\varpi_k \otimes  {\det}_{\mathrm{O}(p,q)}^{k/2}$ descends to $\mathrm{O}(p,q)$.
\item The induced action of $K = \OO(p) \times \OO(q)$ by $\varpi_k \otimes {\det}_{\mathrm{O}(p,q)}^{k/2}$ on the vaccuum vector $\psi_0$ (the constant polynomial $1$) in the Fock model $\mathcal{P}_k$ is given by
\begin{equation*}
\big(\varpi_k \otimes {\det}_{\mathrm{O}(p,q)}^k\big)(k_+, k_-) \big(\psi_0) = {\det}_{\mathrm{O}(q)}(k_-)^k \psi_0.
\end{equation*}
\end{enumerate}
\end{lem}
Applying  items (1),(2),(3) and (4) to the  case in hand, we obtain 
\begin{prop} \label{actionofK}
The action of $K = \OO(n) \times \OO(1)$ on $\mathcal{P}_k$ under  the restriction of the Weil representation twisted by ${\det}_{\mathrm{O}(n,1)}^{k/2}$ is given by 
\begin{equation*}
\big(\varpi_k \otimes {\det}_{\OO(n,1)}^k\big)(k_+, k_-) \big(\varphi \big)(\mathbf{v})  = {\det}_{\OO(1)}(k_-)^k \  \varphi(k_+^{-1} k_-^{-1} \mathbf{v}).
\end{equation*}
\end{prop}

The cohomology groups of interest to us now are the groups
$$H^\ell\big(\mathfrak{so}(n,1) ,\widetilde{K}; \mathcal{P}_k \otimes \det^{k/2}\big).$$

The goal in this subsection is to prove
\begin{thm}\label{maintwo}
\hfill
\begin{enumerate}
\item If $k < n$,
\begin{equation*}
H^\ell \big(\mathfrak{so}(n,1) , \widetilde{\mathrm{O}(n) \times \mathrm{O}(1)}; \rm{Pol}((V \otimes \C)^k)\otimes {\det}^{\frac{k}{2}} \big) = \begin{cases}
\mathcal{R}_k \varphi_k &\text{ if } \ell = k \\
0 &\text{ otherwise.}
\end{cases}
\end{equation*}
\item  If $k=n$,
\begin{equation*}
H^{\ell} \big(\mathfrak{so}(n,1) , \widetilde{\mathrm{O}(n) \times \mathrm{O}(1)}; \mathcal{P}_k \otimes {\det}^{\frac{k}{2}} \big) =
\begin{cases}
\mathcal{R}_n \varphi_n &\text{ if } \ell = n \\
0 &\text{ otherwise.}
\end{cases}
\end{equation*}
\item If $k > n$,
\begin{equation*}
H^{\ell} \big(\mathfrak{so}(n,1) , \widetilde{\mathrm{O}(n) \times \mathrm{O}(1)}; \mathcal{P}_k \otimes {\det}^{\frac{k}{2}} \big) =
\begin{cases}
\text{nonzero} &\text{ if } \ell = n \\
0 &\text{ otherwise.}
\end{cases}
\end{equation*}
\end{enumerate}
\end{thm}


We now prove Theorem \ref{maintwo}. 
Put $K = \OO(n) \times \OO(1)$.  Then from (3) of Lemma \ref{Weildescends} the restriction of the twisted Weil representation of $\widetilde{K}$ descends to $K$ and we have

\begin{equation} \label{twistedactioncohomologygroups}
C^\ell\big(\mathfrak{so}(n,1) ,\widetilde{K}; \mathcal{P}_k \otimes {\det}^{k/2}\big) = C^\ell\big(\mathfrak{so}(n,1) ,K; \mathcal{P}_k \otimes {\det}^{k/2}\big).
\end{equation}
We use the notation $C^\ell(\mathcal{P}_k) = C^\ell(\mathfrak{so}(n,1), \SO(n); \mathcal{P}_k )$.

Note $(\mathrm{O}(n) \times \mathrm{O}(1))/ \SO(n)\cong \Z/2 \times \Z/2$ and apply Proposition \ref{actionofK} to the right-hand side of  Equation \eqref{twistedactioncohomologygroups}  to obtain 
\begin{equation*}
C^\ell(\mathfrak{so}(n,1), K; \mathcal{P}_k \otimes {\det}^{k/2} ) = \big( C^\ell( \mathcal{P}_k ) \otimes {\det}_{\OO(1)}^k \big)^{\Z/2 \times \Z/2}.
\end{equation*}
On the right-hand side of the above equation, we have extended the action of $\SO(n)$ on $(V \otimes \C)^k$ to the action of $\OO(n)$ given by
$$k \varphi(\mathbf{v}) = \varphi(k^{-1} \mathbf{v}).$$

Now recall that $C^\ell( \mathcal{P}_k) = C_+^\ell \oplus C_-^\ell$ and hence
$$C^\ell( \mathcal{P}_k \otimes {\det}_{\OO(1)}^k)^{\Z/2 \times \Z/2} = (C_+^\ell \otimes {\det}_{\OO(1)}^k)^{\Z/2 \times \Z/2} \oplus (C_-^\ell \otimes {\det}_{\OO(1)}^k)^{\Z/2 \times \Z/2}.$$
Since the element $(1,0) \in \Z/2 \times \Z/2$ acts by the element $\iota \otimes \iota$ (see Equation \eqref{iotadef}) and $C_-$ is defined to be the $-1$ eigenspace of the action of $\iota \otimes \iota$, we have $C_-^{\Z/2 \times \Z/2} = 0$ and hence $(C_- \otimes {\det}_{\OO(1)}^k)^{\Z/2 \times \Z/2} = 0$.  Hence
\begin{equation*}
C^\ell(\mathfrak{so}(n,1), K; \mathcal{P}_k \otimes {\det}^{k/2} ) = (C_+^\ell \otimes {\det}_{\OO(1)}^k)^{\Z/2 \times \Z/2}.
\end{equation*}
Hence we have
\begin{equation*}
H^\ell\big(\mathfrak{so}(n,1) ,\widetilde{K}; \mathcal{P}_k \otimes {\det}^{k/2}\big) = (H^\ell(C_+) \otimes {\det}_{\OO(1)}^k)^{\Z/2 \times \Z/2}.
\end{equation*}

\begin{rmk} \label{phiko(1)}
Note that $\varphi_1$ is invariant under $\OO(n)$ and transforms under $\OO(1)$ by ${\det}_{\OO(1)}$. Since $\varphi_k$ is the $k$-fold exterior wedge of $\varphi_1$ with itself, it follows that  $\varphi_k$ is invariant under $\OO(n)$ and transforms under $\OO(1)$ by ${\det}_{\OO(1)}^k$.  This is the reason for twisting the Fock model by ${\det}^{k/2}$.
\end{rmk}

\begin{lem}
$$H^\ell(C_+) = (H^\ell(C_+) \otimes {\det}_{\OO(1)}^k)^{\Z/2 \times \Z/2}.$$
\end{lem}

\begin{proof}
By Theorem \ref{K+vanCvan},
$$H^\ell(C_+) = \begin{cases} \mathcal{R}_k \varphi_k &\text{ if }\ell = k \\
0 &\text{ otherwise. } \end{cases}$$
By Remark \ref{phiko(1)}, $\varphi_k$ transforms by ${\det}_{\OO(1)}^k$ and hence $H^k(C_+)$ transforms by ${\det}_{\OO(1)}^k$ and the lemma follows.

\end{proof}

As an immediate consequence of the previous lemma we have 

\begin{equation*}
H^\ell \big(\mathfrak{so}(n,1) ,\mathrm{O}(n) \times \mathrm{O}(1); \rm{Pol}((V \otimes \C)^k)\otimes {\det}^{\frac{k}{2}} \big) =
\begin{cases}
\mathcal{R}_k \varphi_k \text{ if } \ell = k\\
0 \text{ otherwise.}
\end{cases}
\end{equation*}
and 
statements (1) and (2) of Theorem \ref{maintwo} follow.  

We now prove statement (3).  Thus, we have $k > n$. The vanishing part (for $\ell <n$) of statement (3) follows from the vanishing statement in Theorem \ref{kgeqnvanishing}.  It remains to prove nonvanishing in degree $n$. To this end, we must exhibit classes in $H^n(A)$, where $A$ is defined as in Equation \eqref{defofAcomplex}, which, once twisted by $\det^\frac{k}{2}$, are $\OO(n) \times \OO(1)$-invariant.  By Proposition \ref{Avanishing} we have (this is before we take invariance),
\begin{equation*}
H^\ell(A) = \begin{cases}
0 &\text{ if } \ell \neq n \\
\mathcal{P}_k/(q_1, \ldots, q_n) \vol &\text{ if } \ell = n.
\end{cases}
\end{equation*}
It remains to find a nonzero element which is invariant.  First, a definition and a lemma.  Define ${\det}_+$ to be the determinant of the upper-left $(n \times n)$-block of coordinates.  That is, ${\det_+}$ is the determinant of the $n \times n$-matrix $(z_{\alpha i})_{1 \leq \alpha, i \leq n}$.

\begin{lem} \label{injecttoA}
The map from $\C[w_{n+1}, \ldots, w_k]$ to $\mathcal{P}_k/(q_1, \ldots, q_n)$ sending $f$ to $f {\det}_+$ is an injection.
\end{lem}

\begin{proof}
We will show that this map has kernel zero.  Suppose $f(w_{n+1}, \ldots, w_k) {\det}_+ \in (q_1, \ldots, q_n)$.   Now pass to the quotient where we have divided by the ``off-diagonal" $z_{\alpha i}$.  Then
\begin{align*}
f(w_{n+1}, \ldots, w_k) {\det}_+ \mapsto f(w_{n+1}, \ldots, w_k) z_{11} z_{22} \cdots z_{nn}
q_\alpha \mapsto w_\alpha z_{\alpha \alpha}.
\end{align*}
By assumption, $f \prod_\alpha z_{\alpha \alpha} \in (w_1 z_{11}, w_2 z_{22}, \ldots, w_n z_{nn}) \subset (w_1, \ldots, w_n)$.
Hence
$$f(w_{n+1}, \ldots, w_k) \prod_\alpha z_{\alpha \alpha} \in (w_1, \ldots, w_n)$$
and thus $f = 0$.

\end{proof}

By Lemma \ref{injecttoA} the cohomology classes $f(w_{n+1}, \ldots, w_k) {\det}_+ \vol \in H^n(A)$ are nonzero.  We now check if they are invariant.

Let $f$ be homogenous of degree $a$.  Then for $(k_+, k_-) \in \OO(n) \times \OO(1)$ we have, by Lemma \ref{Weildescends},
$$(k_+ k_-) f {\det}_+ \vol = \det(k_-)^{a+k+n} f {\det}_+ \vol.$$
Thus, if $a+k+n$ is even this element is invariant and Statement (3) is proved.  In fact, we have shown

\begin{prop}
For $k > n$, if $k+n$ is even (resp. odd), then the even (resp. odd) degree polynomials $f \in \C[w_{n+1}, \ldots, w_k]$ inject into\\
$H^n \big(\mathfrak{so}(n,1) , \widetilde{\mathrm{O}(n) \times \mathrm{O}(1)}; \mathcal{P}_k \otimes {\det}^{\frac{k}{2}} \big)$ by the map $f \mapsto f det_+ \vol.$
\end{prop}

%% file: mainthesis.bbl
\begin{thebibliography}{99}
\setlength{\parskip}{1em}


 




\bibitem[BMM1]{BMM1} N. Bergeron, J. Millson and C. Moeglin, \emph{Hodge type theorems for arithmetic manifolds associated to orthogonal groups}, arXiv:1110.3049

\bibitem[BMM2]{BMM2} N. Bergeron, J. Millson and C. Moeglin, \emph{The Hodge conjecture and arithmetic quotients of complex balls}, arXiv:1306.1515

\bibitem[BMR]{BMR} N. Bergeron, J. Millson and J. Ralston, \emph{The relative Lie algebra cohomology of the Weil representation of $\SO(n,1)$}, arXiv:1411.4063




\bibitem[BorW]{BW} A. Borel and N. Wallach, \emph{Continuous Cohomology, Discrete Subgroups, and Representations of Reductive Groups}, Annals of Mathematics Studies \textbf{94}(1980), Princeton University Press.




  




\bibitem[E]{Eisenbud} D. Eisenbud, \emph{Commutative  Algebra with a View Toward Algebraic Geometry}, Graduate Texts in Mathematics,Springer-Verlag \textbf{150}(1995).

\bibitem[Fo]{Folland} G. B. Folland, \emph{Harmonic Analysis in Phase Space}, Princeton University Press, (1989).

\bibitem[FM]{FM} J. Funke and J. Millson, \emph{The geometric theta correspondence for Hilbert modular surfaces}, Duke Math. J. {\bf 163} (2014), no. 1 65--116.

\bibitem[FH]{FH} W. Fulton and J. Harris, \emph{Representation Theory A First Course}, Graduate Texts in Mathematics \textbf{129}(1991), Springer-Verlag.

\bibitem[GW]{GoodmanWallach} R. Goodman and N. Wallach, \emph{Representations and Invariants of the Classical Groups}, Encylopedia of Mathematics and its Applications, Cambridge University Press \textbf{68}(1998).





\bibitem[Ho]{Howe} R. Howe, \emph{Representations and Invariants of the Classical Groups}, The Schur lectures, Israel Mathematical Conference Proceedings  \textbf{8}(1995).





\bibitem[KaV]{KV} M. Kashiwara and M. Vergne, \emph{On the Segal-Shale representations and harmonic polynomials}, Invent. Math. \textbf{44}(1978),1-47.


\bibitem[KM1]{KM1} S. Kudla and J. Millson, \emph{Geodesic cycles and the Weil representation I; quotients of hyperbolic space and Siegel modular forms}, Compos. Math.,  \textbf{45}(1982), 207-271.

\bibitem[KM2]{KM2} S. Kudla and J. Millson, \emph{Intersection numbers of cycles on locally symmetric spaces and Fourier coefficients of holomorphic modular forms in several complex variables}, IHES Pub. Math.  \textbf{71} (1990), 121-172.




\bibitem[Mat]{Matsumura} H. Matsumura, \emph{Commutative Ring Theory}, Cambridge Studies in Advanced Math \textbf{8}, (1986).

\bibitem[Mc]{Mc} J. McCleary, \emph{User's Guide to Spectral Sequences}, Math Lecture Series \textbf{12}, (1985), Publish or Perish, Inc.







\bibitem[W]{Weyl} H. Weyl, \emph{The Classical Groups}, Princeton Mathematical Series. \textbf{1}(1939), Princeton University Press.

\end{thebibliography}
